\documentclass[10pt]{amsart}
\usepackage[T1]{fontenc}
\usepackage{CJKutf8}
\usepackage[margin=1.05in]{geometry}
\usepackage{times}
\usepackage{amsmath,amssymb,mathtools,amsthm}
\usepackage{newtxmath}
\usepackage{mathrsfs}
\usepackage{microtype}
\usepackage{enumitem}
\usepackage{needspace}
\usepackage{xcolor}
\usepackage{tikz}
\definecolor{RoyalPurple}{HTML}{7851A9}
\definecolor{HighlightGold}{HTML}{FFD700}
\usepackage[colorlinks=true,linkcolor=RoyalPurple,citecolor=RoyalPurple,urlcolor=RoyalPurple]{hyperref}
\usepackage[capitalise,nameinlink,noabbrev]{cleveref}
\usepackage{k0-stacks}
\makeatletter
\renewcommand{\paragraph}{\@startsection{paragraph}{4}%
  {\parindent}{\z@}{-\fontdimen2\font}{\normalfont\bfseries}}
\makeatother
\crefname{theorem}{theorem}{theorems}
\crefname{proposition}{proposition}{propositions}
\crefname{lemma}{lemma}{lemmas}
\crefname{corollary}{corollary}{corollaries}
\crefname{definition}{definition}{definitions}
\crefname{remark}{remark}{remarks}
\Crefname{theorem}{Theorem}{Theorems}
\Crefname{proposition}{Proposition}{Propositions}
\Crefname{lemma}{Lemma}{Lemmas}
\Crefname{corollary}{Corollary}{Corollaries}
\Crefname{definition}{Definition}{Definitions}
\Crefname{remark}{Remark}{Remarks}
\crefformat{equation}{(#2#1#3)}
\Crefformat{equation}{(#2#1#3)}
\usepackage[backend=biber, style=alphabetic, maxbibnames=4,
  giveninits=true]{biblatex}
  \providecommand{\bA}{\mathbb{A}} 
\providecommand{\sB}{\mathscr{B}}   
\providecommand{\sC}{\mathscr{C}}  \providecommand{\bC}{\mathbb{C}} 
\providecommand{\sD}{\mathscr{D}}   
 \providecommand{\cE}{\mathcal{E}}  
 \providecommand{\cF}{\mathcal{F}} \providecommand{\bF}{\mathbb{F}} 
  \providecommand{\bG}{\mathbb{G}} 
\providecommand{\sH}{\mathscr{H}}   
\providecommand{\sI}{\mathscr{I}}

 \providecommand{\cL}{\mathcal{L}}  
\providecommand{\sM}{\mathscr{M}}   
   
 \providecommand{\cO}{\mathcal{O}}  
  \providecommand{\bP}{\mathbb{P}} 
  \providecommand{\bQ}{\mathbb{Q}}

 \providecommand{\cU}{\mathcal{U}}  
 \providecommand{\cV}{\mathcal{V}}  
\providecommand{\sW}{\mathscr{W}}   
\providecommand{\sX}{\mathscr{X}}   
\providecommand{\sY}{\mathscr{Y}}   
\providecommand{\sZ}{\mathscr{Z}}  \providecommand{\bZ}{\mathbb{Z}} 

\begin{document}
\begin{CJK*}{UTF8}{gbsn}

\title{The integral $K_0$-ring of the stack of smooth hyperelliptic curves}

\DeclareRobustCommand{\ChineseAuthorName}{曾若凡}
\author[Michael R. Zeng]{Michael Ruofan Zeng \\ \ChineseAuthorName}

\begin{abstract}
We compute the integral $K_0$-ring of the moduli stack $\sH_g$ of smooth hyperelliptic curves of genus $g\ge2$ over a field of characteristic zero. We show that it is a free abelian group of rank $8g+2$ for even $g$ and $8g+4$ for odd $g$. We use the isotypic decomposition of the $K_0$ of a related $\mmu{g+1}$-gerbe to obtain presentations for both even and odd $g$'s from the standard quotient presentation of $\sH_g$. Thus, our technique differs from the $CH^\bullet(\sH_g)$ computations of Edidin--Fulghesu, Fulghesu--Viviani, Di Lorenzo, and the $K_0(\sM_2)$ computation of Edidin--Hu. Using Pagani's description of the inertia stack of $\sH_g$, we apply Riemann--Roch to interpret the rank of $K_0(\sH_g)$ as the number of inertia components. Finally, we compare the $K_0$ and Chow rings of $\sH_g$ through the gamma-filtration and give presentations to the associated graded rings $\operatorname{gr}_{\gamma}K_0(\sH_g)\otimes\bF_p$ for every prime $p$.
\end{abstract}

\maketitle

\begingroup
\setlength{\baselineskip}{11pt}
\tableofcontents
\endgroup
\newpage

\section{Introduction}
The moduli stack $\sH_g$ of smooth hyperelliptic curves of genus $g\ge2$ admits a description as the quotient of an affine space of binary forms, which makes its intersection theory accessible through equivariant methods. Its integral Chow ring is known for all $g$ by work of Edidin--Fulghesu, Fulghesu--Viviani, and Di~Lorenzo \cite{EF09,FV11,DL21}.  In this paper, we compute its integral Grothendieck ring of perfect complexes $K_0(\sH_g) := K_0(\mathsf{Perf}_{\sH_g})$ (equivalent to the Grothendieck ring of vector bundles since $\sH_g$ is smooth), extending the genus-two $K_0(\sH_2)\cong K_0(\sM_2)$ computation of Edidin and Hu \cite[Theorem~1.1]{EH24}.

Let us first review the quotient presentation of $\sH_g$ in terms of binary forms. Let $\bA^{N+1}$ be the affine space of binary forms of degree $N = 2g+2$ in $x_0, x_1$, and let $\Delta$ be the discriminant hypersurface of forms with a repeated root.  A square-free binary form $f$ defines a smooth hyperelliptic curve of genus $g$ with equation $y^2 = f(x_0, x_1)$, which is a double cover of $\bP^1$ branched over the roots of $f$.  Reparametrizing $f$ via the standard $\GL_2$ action where $A\cdot f = f(A^{-1}x)$ preserves the isomorphism type of the curve. Furthermore, a root of unity $\eta\in\mmu{g+1}$ scales a form of degree $N = 2(g+1)$ by $\eta^{-N} = 1$, so the action factors through $\GL_2/\mmu{g+1}$. Arsie and Vistoli \cite{AV04} proved that
\begin{equation}
  \label{eq:Hg-quotient-presentation}
  \sH_g \;\cong\;
\bigl[(\bA^{N+1}\smallsetminus\Delta)\big/(\GL_2/\mmu{g+1})\bigr].
\end{equation}

Edidin and Fulghesu computed the integral Chow ring $CH^\bullet(\sH_g)$ for even $g$ \cite{EF09}. For odd $g$, Di~Lorenzo completed the computation begun by Fulghesu--Viviani \cite{FV11}, correcting a pushforward computation that holds after inverting $2$, but fails integrally \cite[Proposition~5.1 and Corollary~5.3]{DL21}.  The Chow ring presentations are therefore
\begin{equation}\label{eq:chow}
CH^\bullet(\sH_g) \;\cong\;
\begin{cases}
\bZ[c_1, c_2]\big/\bigl(2(2g+1)c_1,\ g(g-1)c_1^2 -
4g(g+1)c_2\bigr), & g \text{ even},\\[2pt]
\bZ[t, c_2, c_3]\big/\bigl(4(2g+1)t,\ 8t^2 - 2g(g+1)c_2,\
2c_3\bigr), & g \text{ odd}.
\end{cases}
\end{equation}
The two parities differ because of the structure of $\GL_2/\mmu{g+1}$. For even $g$, $\GL_2/\mmu{g+1}$ is isomorphic to the special group $\GL_2$, and the computation of Edidin--Fulghesu used standard $\GL_2$-equivariant techniques. For odd $g$, however, $\GL_2/\mmu{g+1}$ is isomorphic to $\bG_m\times\PGL_2$, and the non-special factor $\PGL_2$ complicates the integral Chow computation. Di~Lorenzo instead used another presentation of $\sH_g$ as a $\GL_3\times\bG_m$-quotient, where equivariant techniques again apply \cite{DL21}. In contrast, we directly work with the presentation \eqref{eq:Hg-quotient-presentation} in both parities: we compute the $K_0$-ring of the $\mmu{g+1}$-gerbe \begin{equation}
\sX_g := [(\bA^{N+1}\smallsetminus\Delta)/\GL_2]
\end{equation} over $\sH_g$ and extract $K_0(\sH_g)$ as its weight-zero summand for the $\mmu{g+1}$-characters.

The generators of our presentations of $K_0(\sH_g)$ are classes of the following vector bundles on $\sH_g$.  Let $\pi\colon \sC\to\sH_g$ be the universal curve, $\iota$ the hyperelliptic involution on $\sC$, and $\sW\subset\sC$ the divisor of Weierstrass points.  The bundles in concern come from the univeral family of rational curves $\sC/\iota$, and they depend on the parity.  For even $g$, $\sC/\iota\cong\bP(\cV)$ for the bundle $\cV$ induced by the standard two-dimensional representation under the isomorphism $\GL_2/\mmu{g+1}\cong\GL_2$.  For odd $g$, $\sC/\iota$ embeds in the projectivization of the rank-three bundle
\[
\cE \;:=\; \pi_*\,\omega_{\sC/\sH_g}^{-1}(\sW).
\]
In this case, we also use the line bundle
\[
\cL \;:=\; \pi_*\,\omega_{\sC/\sH_g}^{\otimes(g+1)/2}
\bigl(\tfrac{1-g}{2}\sW\bigr),
\]
which generates $\Pic(\sH_g)\cong\bZ/(8g+4)$ \cite[Section~6]{DL21}. Indeed, the generators of the Chow presentations \eqref{eq:chow} are Chern classes of these bundles: $c_i=c_i(\cV)$ for even $g$ \cite{EF09}, and $t=c_1(\cL)$ and $c_i=c_i(\cE)$ for odd $g$ \cite[Theorem~6.2]{DL21}.

Our first presentation of $K_0(\sH_g)$ uses $K_0$-classes of these vector bundles. Write $\nu := [\cV]$ and $\lambda := [\det\cV]$ for even $g$, and $\sigma := [\cE]$ and $\lambda := [\cL]$ for odd $g$. In both parities, $\lambda$ is the class of a generator of $\Pic(\sH_g)$ (\Cref{rem:pic}).  For $k\ge0$, define the polynomials
\begin{equation}\label{eq:famdef}
P_k(X, Y) \;=\; \sum_{j=0}^{\lfloor k/2\rfloor}
(-1)^j\binom{k-j}{j}\,Y^{\,j}X^{\,k-2j},
\qquad
Q_k(Z) \;=\; \sum_{j=0}^{k}
(-1)^j\binom{2k-j}{j}\,(Z+1)^{k-j}.
\end{equation}
They satisfy $P_k(\nu,\lambda)=[\Sym^k\cV]$, while $Q_k(\sigma)$ is the class of the bundle induced by the $(2k+1)$-dimensional irreducible representation of $\PGL_2$ under $\GL_2/\mmu{g+1}\cong\bG_m\times\PGL_2$. For example, $P_2(\nu,\lambda)=\nu^2-\lambda$ and $Q_2(\sigma)=\sigma^2-\sigma-1$.  With these polynomials, we can write down the relations rather compactly.

\begin{theorem}\label{thm:intromain}
Let $g\ge2$ and work over a field of characteristic zero. Then,
\begin{align}\label{eq:K0-Hg-bundle-class}
K_0(\sH_g) \;\cong \; \bZ[\nu, \lambda^{\pm1}]\big/\bigl(
&1 + \nu\lambda^{1-g}\,P_{2g-1}(\nu, \lambda)
- \lambda^{2-2g}\,P_{4g}(\nu, \lambda),\\[-1pt]
&\nu + \nu\lambda^{2-g}\,P_{2g-2}(\nu, \lambda)
- \lambda^{3-2g}\,P_{4g-1}(\nu, \lambda)\bigr)
&& (g \text{ even}),\notag\\
K_0(\sH_g) \;\cong \; \bZ[\sigma, \lambda^{\pm1}]\big/\bigl(
&1 + \lambda^2\bigl(Q_g(\sigma) + Q_{g-1}(\sigma)\bigr)
- \lambda^4\,Q_{2g}(\sigma),\\[-1pt]
&(\sigma{+}1)\bigl(1 + \lambda^2\,Q_{g-1}(\sigma)\bigr)
- \lambda^4\bigl(Q_{2g}(\sigma) + Q_{2g-1}(\sigma)\bigr)\bigr)
&& (g \text{ odd}).\notag
\end{align}
In particular, $K_0(\sH_g)$ is a free abelian group of rank $8g+2$ for even $g$ and $8g+4$ for odd $g$.
\end{theorem}

We give a second presentation in terms of \emph{$K$-theoretic Chern classes}. The $K$-theoretic Chern classes $c_i^K(\mathcal E)$ are defined using exterior powers of $\mathcal E^\vee$ and are closely related to fundamental classes in $K$-theory, as recalled in \Cref{subsec:kchern}. For even $g$, let
\[
\delta_1:=c_1^K(\det\cV)=1-\lambda^{-1},
\qquad
\delta_2:=c_2^K(\cV)=1-\nu\lambda^{-1}+\lambda^{-1},
\]
and for odd $g$, let
\[
\tau_1:=c_1^K(\cL)=1-\lambda^{-1},
\qquad
\kappa_2:=c_2^K(\cE)=3-\sigma.
\]

\begin{theorem}\label{thm:introfc}
Let $g\ge2$ and work over a field of characteristic zero. For even $g$, there is an isomorphism
\begin{equation}\label{eq:K0-Hg-chern-class}
K_0(\sH_g)\cong
\frac{\bZ[\delta_1,\delta_2]}{\left(
\begin{aligned}
&\textcolor{RoyalPurple}{-(4g+2)\delta_1}+\sum\nolimits_{i=2}^{2g+2}(-1)^i\tbinom{2g+2}{i}\delta_1^i\\
&\quad+\sum\nolimits_{\substack{i,j\ge0,\ i+j\le4g\\i+2j\ge2}}(-1)^{i+j}\delta_1^i\delta_2^j
\Bigl[-\tbinom{i+j}{j}\tbinom{4g+j+1}{i+2j+1}+\sum\nolimits_{a=0}^i\tbinom{i-a+j}{j}
\bigl(\tbinom{g+1}{a}\tbinom{2g+j+1}{i-a+2j+1}
+\tbinom{g+2}{a}\tbinom{2g+j-1}{i-a+2j+1}\bigr)\Bigr],\\[2pt]
&\textcolor{RoyalPurple}{g(g-1)\delta_1^2-4g(g+1)\delta_2}
+\sum\nolimits_{i=3}^{2g+1}(-1)^i\tbinom{2g+1}{i}\delta_1^i
-\delta_2\sum\nolimits_{i=1}^{2g+1}(-1)^i\tbinom{2g+1}{i}\delta_1^i\\
&\quad+\sum\nolimits_{\substack{i,j\ge0,\ i+j\le4g\\i+2j\ge3}}(-1)^{i+j}\delta_1^i\delta_2^j
\Bigl[\tbinom{i+j}{j}\tbinom{4g+j}{i+2j}-\sum\nolimits_{a=0}^i\tbinom{i-a+j}{j}
\bigl(\tbinom{g+1}{a}\tbinom{2g+j}{i-a+2j}
+\tbinom{g+2}{a}\tbinom{2g+j-2}{i-a+2j}\bigr)\Bigr],\\[2pt]
&\textcolor{RoyalPurple}{(4g+2)\delta_1}+\sum\nolimits_{i=2}^{4g+2}(-1)^{i+1}\tbinom{4g+2}{i}\delta_1^i
\end{aligned}
\right)}.
\end{equation}
For odd $g$, we have
\begin{equation}\label{eq:odd-chern-binomial}
K_0(\sH_g)\cong
\frac{\bZ[\tau_1,\kappa_2]}{\left(
\begin{aligned}
&\textcolor{RoyalPurple}{-(8g+4)\tau_1}+\sum\nolimits_{i=2}^4(-1)^i\tbinom4i\tau_1^i
-\sum\nolimits_{j=1}^{2g}(-1)^j\bigl(\tbinom{2g+j+1}{2j+1}+\tbinom{2g+j}{2j+1}\bigr)\kappa_2^j\\
&\quad+\sum\nolimits_{\substack{0\le i\le2,\ 0\le j\le2g\\i+2j\ge2}}(-1)^{i+j}\tbinom2i
\bigl(\tbinom{g+j+1}{2j+1}+2\tbinom{g+j}{2j+1}+\tbinom{g+j-1}{2j+1}\bigr)\tau_1^i\kappa_2^j,\\[2pt]
&\textcolor{RoyalPurple}{8\tau_1^2-2g(g+1)\kappa_2}
+2\sum\nolimits_{i=3}^4(-1)^i\tbinom4i\tau_1^i
-\kappa_2\sum\nolimits_{i=1}^4(-1)^i\tbinom4i\tau_1^i+\sum\nolimits_{j=2}^{2g}(-1)^j\bigl(\tbinom{2g+j}{2j}+\tbinom{2g+j-1}{2j}\bigr)\kappa_2^j\\
&\quad-\sum\nolimits_{\substack{0\le i\le2,\ 0\le j\le2g\\i+2j\ge3}}(-1)^{i+j}\tbinom2i
\bigl(\tbinom{g+j}{2j}+2\tbinom{g+j-1}{2j}+\tbinom{g+j-2}{2j}\bigr)\tau_1^i\kappa_2^j,\\[2pt]
&\textcolor{RoyalPurple}{(8g+4)\tau_1}+\sum\nolimits_{i=2}^{8g+4}(-1)^{i+1}\tbinom{8g+4}{i}\tau_1^i
\end{aligned}
\right)}.
\end{equation}
\end{theorem}

For $g=2$, we list the fully expanded presentations in \Cref{app:edidinhu}. The ring $K_0(\sH_g)$ is not graded, so each relation we find is inhomogeneous, in the form \[\textcolor{RoyalPurple}{\text{lowest degree part}} \quad + \quad \text{ higher order terms.}\] The lowest degree parts of \eqref{eq:K0-Hg-chern-class} and \eqref{eq:odd-chern-binomial} are colored in purple. We observe that each of these terms has the same form as a familiar relation in $CH^\bullet(\sH_g)$, as expected from the gamma-filtration (also recalled in \Cref{subsec:kchern}).

We remark that the presentations above require $\sH_g$ to be defined over characteristic zero. In characteristic $3$, for example, the invariant linear subspace $\Bbbk\cdot\{x_0^6,x_0^3x_1^3,x_1^6\}$ of binary sextics lies entirely in the discriminant, because $(ax_0^2+bx_0x_1+cx_1^2)^3=a^3x_0^6+b^3x_0^3x_1^3+c^3x_1^6$. Its Koszul resolution would yield an additional relation in $K_0(\sH_2)$, which is however not a relation in the above presentation even with rational coefficients (\Cref{rem:basefield}). The characteristic-$3$ example shows that the relations in rational $K_0$ can change with the characteristic. For comparison, Di~Lorenzo's presentation of $CH^\bullet(\sH_g)$ for odd $g$ holds over fields of characteristic different from $2$ \cite[Theorem~6.2]{DL21}.

\subsection{Proof outline} To compute $K_0(\sX_g)$ as a quotient of $R(\GL_2)$, we obtain two relations $F_g,G_g$ from the double-root locus by equivariant localization methods as used by Edidin--Hu \cite[Section~3.2]{EH24}. To prove that they generate the whole kernel, one might follow Edidin--Hu and use equivariant envelopes to exhaust all pushforwards from the discriminant. Since that computation turns out to be cumbersome for general $g$, we instead study the quotient $A_g:=R(\GL_2)/(F_g,G_g)$. We show that $A_g$ is free of finite rank over $\bZ$ and that $A_g\otimes\overline{\bQ}$ is reduced, so elements of $A_g$ are determined by their values at the points of $\Spec(A_g\otimes\overline{\bQ})$. Each such point is represented by the eigenvalues of an automorphism fixing a square-free binary form. Restriction to the stabilizer of that form shows that every element of $\ker(A_g\to K_0(\sX_g))$ vanishes at the point, and hence the kernel is zero (\Cref{thm:eqpres}).

Having an integral presentation of $K_0(\sX_g)$ brings us close to the finish line. The central $\mmu{g+1}\subset\GL_2$ induces a grading on $K_0(\sX_g)$ by its characters $\mmu{g+1}^\vee$ (\Cref{thm:isotypic}), and we can straightforwardly recover $K_0(\sH_g)$ as the weight-zero subring:
\[K_0(\sX_g)\cong \bigoplus_{\chi\in \mmu{g+1}^\vee} K_0(\sX_g)_\chi, \qquad K_0(\sX_g)_0 \cong K_0(\sH_g).\]
A bit of commutative algebra then yields the desired presentations.

We remark that the Chow ring admits no such decomposition technique in general. For example, $CH^1(\sB\bG_m)$ and $CH^1(\sB(\bG_m/\mmu n))$ are both infinite cyclic, but pullback along the quotient map sends a generator of the latter to $n$ times a generator of the former. For $n>1$, the image of pullback is therefore $n\bZ\subset\bZ$, which is not a direct summand because $\bZ/n\bZ$ is torsion.

To compute the rank of $K_0(\sH_g)$, let $\xi=[\cU]$ and $\psi=[\det\cU]$ for the rank-two bundle $\cU$ on $\sX_g$ associated to the standard representation of $\GL_2$. For even $g$, multiplication by the unit $\psi$ of $\mmu{g+1}$-weight two identifies summands whose weights differ by two and gives rank $8g+2$ in each summand $K_0(\sX_g)_j$. For odd $g$, multiplication by the weight-one class $\xi$ is injective on each odd-weight summand since $A_g$ is reduced, and each even-weight component of $A_g/(\xi)\cong\bZ[\psi]/(\psi^N-1)$ is free of rank four. Together with the total rank, this gives rank $8g$ in odd weights and $8g+4$ in even weights, the latter of which includes the weight-zero summand $K_0(\sH_g)$ (\Cref{prop:characterranks}).

For odd $g$, these ranks give a $K$-theoretic proof that the universal conic $\sC/\iota\to\sH_g$ is not the projectivization of a rank-two vector bundle, recovering the nonvanishing of its Brauer class \cite[Section~8]{DLP21}. Indeed, $K_0(\sC/\iota)$ has rank $16g+4$, whereas the projective bundle formula would give rank $16g+8$ (\Cref{cor:universalconic}).

\subsection{The gamma-filtration, Riemann--Roch, and $\operatorname{Spec}K_0(\sH_g)$} We move onto comparing $K_0(\sH_g)$ with the Chow ring through the gamma-filtration and through their prime spectra over $\Spec\bZ$. The gamma-filtration agrees with weighted degree in the Chern-class generators (\Cref{prop:gamma}), so the lowest-degree parts of the $K_0$-relations give the following quotients of the Chow ring.

\begin{theorem}\label{thm:introleading}
Let $g\ge2$ and work over a field of characteristic zero. There are surjective homomorphisms of graded rings
\[
\begin{aligned}
CH^\bullet(\sH_g)
&\twoheadrightarrow\operatorname{gr}_\gamma K_0(\sH_g)
&& (g\text{ even}),\\
CH^\bullet(\sH_g)/(c_3)
&\twoheadrightarrow\operatorname{gr}_\gamma K_0(\sH_g)
&& (g\text{ odd}),
\end{aligned}
\]
sending $c_1,c_2$ to the images of $\delta_1,\delta_2$ for even $g$ and $t,c_2$ to the images of $\tau_1,\kappa_2$ for odd $g$.
\end{theorem}

For odd $g$, the equality $c_3^K(\cE)=0$ explains the vanishing of $c_3$ in the associated graded ring (\Cref{lem:cEchern}). Since additional lowest-degree relations can arise from combinations of these relations, these maps are only surjective in general.

Over an algebraically closed field of characteristic zero, the Riemann--Roch isomorphism $K_0(\sH_g)\otimes\overline{\bQ}\cong CH^\bullet(\sI\sH_g)\otimes\overline{\bQ}$ (\Cref{thm:inertiachern}) expresses the rank of $K_0(\sH_g)$ as the sum of the dimensions of the rational Chow rings of the inertia components. We make this precise in the following theorem.

\begin{theorem}\label{thm:introlengths}
Let $\mathscr X$ be a smooth connected separated Deligne--Mumford stack of finite type over an algebraically closed field of characteristic zero, with quasi-projective coarse moduli space and the resolution property. Assume that $K_0(\mathscr X)$ is free abelian of finite rank $r$. Then every geometric fiber of $\Spec K_0(\mathscr X)\to\Spec\bZ$ has length $r$, where
\[
r=\sum_{\mathscr Z\in\pi_0(\sI\mathscr X)}
\dim_{\bQ}\bigl(CH^\bullet(\mathscr Z)\otimes\bQ\bigr)
\ge\#\pi_0(\sI\mathscr X).
\]
Equality holds if and only if $K_0(\mathscr X)\otimes\bQ$ is reduced. In this case,
\[
CH^\bullet(\mathscr Z)\otimes\bQ\cong\bQ,
\]
for every connected component $\sZ \subseteq \sI\sX$, and $\Spec(K_0(\mathscr X)\otimes\overline{\bQ})$ has one reduced point corresponding to each of these inertia components. 
\end{theorem}

For $\sH_g$, Pagani classified the inertia components and observed that each has rational Chow ring $\bQ$ \cite[Theorem~4.2 and Remark~5.2]{Pag12}. Applying \Cref{thm:introlengths} gives an independent explanation of the rank formula
\[
r_g=\#\pi_0(\sI\sH_g)
=\begin{cases}8g+2,&g\text{ even},\\8g+4,&g\text{ odd},
\end{cases}
\]
which is positive evidence that our presentations are correct. The odd-genus case has two more inertia components because $-1\in \mmu{g+1}$ exactly when $g$ is odd. In contrast to $\Spec K_0(\sH_g)$, whose geometric fibers all have length $r_g$, the fibers of $\Spec CH^\bullet(\sH_g)$ have length one over $\bQ$ and over $\bF_p$ for $p\nmid2g(g+1)(2g+1)$, and have positive dimension at every remaining prime (\Cref{prop:chowfibers}).

\paragraph{Question.} For a scheme or algebraic stack $\sX$, how can the prime ideals of $K_0(\sX)$ and $CH^\bullet(\sX)$ be interpreted in terms of the geometry of $\sX$?

We refine the comparison in \Cref{thm:introleading} by computing the associated graded ring after reduction modulo each prime. For $\Bbbk=\bQ$ or $\Bbbk=\bF_p$ with $p$ prime, let $\mathfrak m=\ker(\operatorname{rank}\colon K_0(\sH_g)\otimes \Bbbk\to \Bbbk)$, and let $F$ be the image of the gamma-filtration on $K_0(\sH_g)\otimes \Bbbk$. We prove that $\operatorname{gr}_F(K_0(\sH_g)\otimes \Bbbk)$ is a quotient of $CH^\bullet(\sH_g)\otimes \Bbbk$ of dimension $\dim_\Bbbk\bigl(K_0(\sH_g)\otimes \Bbbk\bigr)_{\mathfrak m}$, with $c_3$ mapping to zero for odd $g$ (\Cref{prop:augmentation}). For $\Bbbk=\bF_p$, we have
\[
\dim_{\bF_p}\bigl(K_0(\sH_g)\otimes\bF_p\bigr)_{\mathfrak m}=
\begin{cases}
p^{v_p(2g+1)}+\dfrac{p^{v_p(g(g+1))}-1}{2},&p\ne2,\\[5pt]
2^{v_2(g)+1}+2,&p=2,\ g\text{ even},\\[2pt]
2^{v_2(g+1)+1}+4,&p=2,\ g\text{ odd},
\end{cases}
\]
where $v_p(n)$ is the exponent of $p$ in $n$ (\Cref{prop:local-lengths}). Note that we are reducing modulo $p$ at the level of the ring $K_0(\sH_g)$, where $\sH_g$ is still defined over a field of characteristic zero. 

Along these lines, we give a complete description of $\operatorname{gr}_F(K_0(\sH_g)\otimes\bF_p)$ for every $g$ and $p$ (\Cref{thm:graded-prime-coefficients}). For odd $p$, it is either a truncated polynomial ring generated by a single Chern class, or $\bF_p$ alone. For $p=2$, both $\delta_1,\delta_2$ in even genus and $\tau_1,\kappa_2$ in odd genus are needed. For example,
\begin{equation}\label{eq:K0-H3-F2}
\operatorname{gr}_F\bigl(K_0(\sH_3)\otimes\bF_2\bigr)
\cong\bF_2[\tau_1,\kappa_2]/(\tau_1^4,\tau_1^2\kappa_2,\kappa_2^5).
\end{equation}
This is a quotient of $\bigl(CH^\bullet(\sH_3)/(c_3)\bigr)\otimes\bF_2\cong\bF_2[t,c_2]$ of length $12$, corresponding to the twelve points of $\Spec(K_0(\sH_3)\otimes\overline{\bQ})$ that specialize to the kernel of the rank homomorphism modulo $2$.

\paragraph{Structure of the paper}
The rest of this paper is organized as follows. In \Cref{sec:preliminaries}, we review techniques from equivariant $K$-theory, isotypic decomposition of $K_0$, the comparison between bundle classes and $K$-theoretic Chern classes, and Riemann--Roch for stacks. In \Cref{sec:equivariant}, we compute $K_0(\sX_g)$ by localization along the discriminant and prove that the relations $F_g,G_g$ generate the kernel of $R(\GL_2)\to K_0(\sX_g)$.  In \Cref{sec:khg}, we take the weight-zero subring and determine $K_0(\sH_g)$ (\Cref{thm:intromain,thm:introfc}) in both parities. We then study $\operatorname{Spec}K_0(\sH_g)\otimes \overline{\mathbb Q}$.  In \Cref{sec:inertia}, we prove \Cref{thm:introlengths} from Riemann--Roch and relate it to Pagani's classification of automorphisms of smooth hyperelliptic curves, recovering the rank of $K_0(\sH_g)$ as the number of inertia components. We then compare the Chow ring with the gamma-filtration and give presentations of $\operatorname{gr}_F(K_0(\sH_g)\otimes\bF_p)$ for every prime $p$. In \Cref{sec:future-directions}, we discuss potential applications of the technique of isotypic decomposition in $K_0$ as future directions. \Cref{app:expanded} contains expanded relations of $K_0(\sH_g)$ for small $g$.

\paragraph{Conventions}
Unless another characteristic is specified, the base field for $\sH_g$ has characteristic zero. We use integral coefficients for $K_0$-rings and Chow rings. Stacks and stacky loci are denoted by script capital letters, schemes by plain Latin capital letters, sheaves and bundles by calligraphic letters.  Classes in $K_0$-rings are lowercase Greek letters, Chow classes are lowercase Latin, and the lowercase Greek letters $\rho$, $\zeta$, $\eta$ also denote scalars in certain contexts.  By the \emph{$K_0$-ring} of a quotient stack $\mathscr X = [X/G]$, we mean the Grothendieck ring $K_0(\mathscr X)$ of perfect complexes, which agrees with the Grothendieck ring $K_0(X, G)$ of $G$-equivariant vector bundles on the smooth scheme $X$ \cite{Tho87}.

\paragraph{Acknowledgments}
The author thanks Hannah K. Larson, Jarod Alper, and Emma Bishop for many helpful conversations and encouragement.

\paragraph{AI disclosure}
The author developed most mathematical contributions and initial drafts prior to June 2026, independently of AI. The author used ChatGPT to find references. At later stages of this project, the author used OpenAI Codex 5.6 Sol and Astra 6 to proofread the manuscript and to check for errors and gaps. Astra 6 found no significant gaps but found short arguments for what has become \Cref{prop:masterflat} through \Cref{thm:eqpres} and suggested direction (4) in \Cref{sec:future-directions}. Codex was also used to generate Macaulay2 \cite{M2} code to confirm the presentations in small genera (\Cref{app:edidinhu}), some calcultions in \Cref{app:chernproofs}, and to create \Cref{fig:g3-specialization}. In accordance with the Leiden Declaration \cite{LeidenDeclaration2026}, the author retains full responsibility for the mathematical content, computational verifications, figures, and references.

\section{Preliminaries}
\label{sec:preliminaries}

We recall some facts about equivariant $K_0$-rings including the localization sequence and the equivariant localization formula. We then review the technique of isotypic decomposition for $K_0$ of a gerbe. To compare $K_0$ with Chow rings, we review $K$-theoretic Chern classes, the gamma-filtration, and Riemann--Roch for stacks. We conclude this section with some polynomial identities needed to express the relations in bundle classes.

\subsection{\texorpdfstring{Equivariant $K$-theory and localization}{Equivariant K-theory and localization}}

For a smooth stack $\sX$ with the resolution property \cite[Definition~7.2.9 and Proposition~7.2.10]{Alp26}, the Grothendieck ring of perfect complexes $K_0(\sX):= K_0(\mathsf{Perf}_{\sX})$ can be described in terms of vector bundles. As an abelian group, it is generated by their isomorphism classes $[\cE]$, subject to $[\cE]=[\cE']+[\cE'']$ for every short exact sequence $0\to\cE'\to\cE\to\cE''\to0$. Tensor product induces the structure of a unital ring on $K_0(\sX)$ with $[\cE][\cF]=[\cE\otimes\cF]$ and $1 = [\cO_{\sX}]$. Since pullback preserves exact sequences of vector bundles and tensor products, a morphism $f\colon\sX\to\sY$ of such stacks induces a ring homomorphism $f^*\colon K_0(\sY)\to K_0(\sX)$, sending $[\cE]$ to $[f^*\cE]$.

Let $G$ be an affine algebraic group acting on a smooth quasi-projective variety $X$ over $\Bbbk$. Vector bundles on the quotient stack $[X/G]$ correspond to $G$-equivariant vector bundles on $X$, which we denote by $\mathsf{VB}^G_{\sX}$ \cite[Section~7.1.2 and Exercise~7.1.11]{Alp26}. Defining $K_0(X,G)$ by the same relations for equivariant vector bundles therefore gives $K_0([X/G])\cong K_0(X,G):= K_0(\mathsf{VB}^G_{\sX})$. For $X=\Spec\Bbbk$, $G$-equivariant vector bundles on $\Spec \Bbbk$ are finite-dimensional $G$ representations, so $K_0(\sB G)$ is isomorphic to the representation ring $R(G)$, by definition. Pullback along $[X/G]\to\sB G$ makes $K_0(X,G)$ an $R(G)$-algebra, and restriction to a subgroup $H\subset G$ induces a ring homomorphism $K_0(X,G)\to K_0(X,H)$ \cite[Section~2.2.1]{Mer05}.

To define proper pushforward, we also use $\mathrm G_0(X,G)$, the Grothendieck group of $G$-equivariant coherent sheaves on a possibly singular $G$-scheme $X$, with relations coming from short exact sequences. When $X$ is smooth and quasi-projective, finite equivariant locally free resolutions give $K_0(X,G)\cong\mathrm G_0(X,G)$ \cite[Theorem~5.7]{Tho87}. A proper $G$-equivariant morphism $f\colon X\to Y$ induces $f_*\colon\mathrm G_0(X,G)\to\mathrm G_0(Y,G)$ by
\[
f_*[\cF]=\sum_{i\ge0}(-1)^i[R^if_*\cF].
\]
Tensoring with vector bundles makes $\mathrm G_0(X,G)$ a $K_0(X,G)$-module, and there is a  projection formula
\[
f_*(f^*[\cE]\cdot[\cF])=[\cE]\cdot f_*[\cF]
\]
for an equivariant vector bundle $\cE$ on $Y$ and an equivariant coherent sheaf $\cF$ on $X$.

\begin{theorem}[Localization sequence, {\cite[Theorem~2.7]{Tho87}}]\label{thm:localizationsequence}
Let $G$ be an affine algebraic group acting on a smooth quasi-projective variety $X$, and let $Z\subset X$ be a $G$-invariant closed subscheme with open complement $U$. The sequence
\begin{equation}\label{eq:localization}
\mathrm G_0(Z,G)\longrightarrow K_0(X,G)\longrightarrow K_0(U,G)\longrightarrow0
\end{equation}
is exact.
\end{theorem}

\begin{theorem}[Homotopy invariance {\cite[Theorem~4.1]{Tho87}}]\label{thm:homotopyinvariance}
Let $G$ be an affine algebraic group acting on a smooth quasi-projective variety $X$ and linearly on $\bA^n$. For the diagonal action on $X\times\bA^n$, the projection $p\colon X\times\bA^n\to X$ induces an isomorphism
\[
p^*\colon K_0(X,G)\xrightarrow{\cong}K_0(X\times\bA^n,G).
\]
\end{theorem}

Let $X$ be a smooth quasi-projective $\GL_2$-variety and $T\subset\GL_2$ the diagonal torus. By Merkurjev's theorem \cite[Proposition~31]{Mer05}, restriction induces an isomorphism
\[
K_0(X,\GL_2)\otimes_{R(\GL_2)}R(T)\xrightarrow{\cong}K_0(X,T),
\]
with the Weyl group $\mathfrak S_2$ acting on the $R(T)$ factor by interchanging the classes $\alpha,\beta$ of the characters of weights $(1,0)$ and $(0,1)$. Since $R(T)$ is free over $R(\GL_2)$ with basis $1,\alpha$, restriction $K_0(X,\GL_2)\to K_0(X,T)$ is injective. If $K_0(X,\GL_2)$ has no $2$-torsion, restriction induces an isomorphism
\[
K_0(X,\GL_2)\cong K_0(X,T)^{\mathfrak S_2}.
\]
Indeed, the nontrivial Weyl element sends $u+\alpha v$ to $u+(\alpha+\beta)v-\alpha v$ for $u,v\in K_0(X,\GL_2)$. Since $\alpha+\beta$ is the restriction of the standard representation, comparison of the coefficients of $\alpha$ gives $2v=0$ and hence $v=0$.

For example, take $X=\Spec \Bbbk$. Let $\cU$ be the standard representation of $\GL_2$, and let $\xi=[\cU]$ and $\psi=[\det\cU]$. Restriction to $T$ gives
\begin{equation}\label{eq:representationring}
R(\GL_2)\cong\bZ[\alpha^{\pm1},\beta^{\pm1}]^{\mathfrak S_2}\cong\bZ[\xi,\psi^{\pm1}],
\qquad \xi=\alpha+\beta,\quad\psi=\alpha \beta.
\end{equation}

To evaluate pushforwards, we use the following form of equivariant localization.

\begin{theorem}[Equivariant localization formula, {\cite{Tho92,EH24}}]\label{prop:toruslocalization}
Let $T$ be a split torus acting on a smooth projective scheme $X$ whose fixed locus $X^T$ consists of finitely many rational points, and let $S\subset R(T)$ be the multiplicative set generated by the $K$-theoretic Euler classes of the nontrivial characters.  Then the pushforward $S^{-1}K_0(X^T, T)\to S^{-1}K_0(X, T)$ is an isomorphism, and for every $\beta\in K_0(X, T)$ and every proper $T$-equivariant map $\pi\colon X\to Y$ to a smooth projective scheme whose fixed locus also consists of finitely many rational points,
\[
\pi_*\beta \;=\; \sum_{x\in X^T}
\frac{\beta|_x}{\lambda_{-1}\bigl(\mathcal{N}_x^\vee X\bigr)}\,
\bigl[\cO_{\pi(x)}\bigr]_T
\qquad\text{in }S^{-1}K_0(Y,T),
\]
where $\lambda_{-1}(\mathcal{N}^\vee) = \sum_i(-1)^i[\wedge^i\mathcal{N}^\vee]$ and $[\cO_y]_T\in K_0(Y, T)$ is the class of the fixed point $y$. If, in addition, $K_0(X, T)$ is a free $R(T)$-module, then restriction to the fixed locus is injective on $K_0(X, T)$.
\end{theorem}

\begin{proof}
For the inclusion $i\colon X^T\hookrightarrow X$, Thomason's localization theorem \cite{Tho92} makes $i_*$ an isomorphism after localizing at $S$. By the self-intersection formula, as stated in Edidin--Hu \cite[Propositions~2.3--2.4]{EH24},
\[
\beta=i_*\left(\frac{i^*\beta}{\lambda_{-1}(\mathcal N^\vee_{X^T}X)}\right)
=\sum_{x\in X^T}\frac{\beta|_x}{\lambda_{-1}(\mathcal N_x^\vee X)}[\cO_x]_T.
\]
Applying $\pi_*$ and using $\pi_*[\cO_x]_T=[\cO_{\pi(x)}]_T$ gives the asserted formula, as in \cite[Section~3.2]{EH24}. The Euler classes in the denominator are units after localizing at $S$, so $i^*$ is also an isomorphism there. If $K_0(X,T)$ is free over $R(T)$, it embeds into $S^{-1}K_0(X,T)$, and hence $i^*$ is already injective on $K_0(X,T)$.
\end{proof}

\subsection{Isotypic decomposition of $K_0$ of a gerbe}

A morphism of stacks $\sX\to\sY$ is a gerbe if, for every scheme $U\to\sY$, the projection $\sX\times_{\sY}U\to U$ admits a section after an fppf cover of $U$, and any two sections are fppf locally isomorphic. A gerbe banded by a commutative group scheme $\mu$ over $\sY$, or a $\mu$-gerbe, is a gerbe with identifications of its relative automorphism sheaves with the pullbacks of $\mu$, compatible with pullback and isomorphisms of objects \cite[Definition~7.4.16]{Alp26}. For example, if a finite diagonalizable central subgroup $\mu$ of an affine group $G$ acts trivially on a $G$-scheme $X$, then $[X/G]\to[X/(G/\mu)]$ is a $\mu$-gerbe \cite[Proposition~7.4.55 and Exercise~7.4.58(b)]{Alp26}. Every $G$-equivariant sheaf on $X$ splits into $\mu$-eigensheaves, so $K_0(X,G)$ splits as a direct sum graded by characters of $\mu$, with weight-zero summand $K_0(X,G/\mu)$. The following is the equivariant form of the decomposition for banded gerbes \cite[Proposition~2.2.1.6]{Lie07}.

\begin{theorem}[Isotypic decomposition]\label{thm:isotypic}
Let $G$ be an affine algebraic group, let $\mu\subseteq Z(G)$ be a finite diagonalizable central subgroup with character group $\mu^\vee$, and let $X$ be a smooth $G$-scheme on which $\mu$ acts trivially.  Then:
\begin{enumerate}[leftmargin=2.4em]
\item every $G$-equivariant coherent sheaf $\mathcal{F}$ on $X$ decomposes canonically as $\mathcal{F} \cong  \bigoplus_{\chi\in\mu^\vee}\mathcal{F}_\chi$, where $\mu$ acts on $\mathcal{F}_\chi$ through $\chi$.
\item the induced decomposition $K_0(X, G) = \bigoplus_{\chi\in\mu^\vee}K_0(X, G)_\chi$ is a ring grading by $\mu^\vee$.
\item descent of sheaves of trivial $\mu$-character induces a ring isomorphism $K_0(X, G)_0 \cong K_0(X, G/\mu)$.
\end{enumerate}
\end{theorem}

\begin{proof}
By \cite[Theorem~4.7]{BS21}, the character decomposition of sheaves on the $\mu$-gerbe $[X/G]\to[X/(G/\mu)]$ preserves coherent sheaves and vector bundles. Taking tensor products amounts to adding characters. It therefore induces a $\mu^\vee$-grading on $K_0(X,G)$, and descent for sheaves of trivial character gives the pullback isomorphism $K_0(X,G/\mu)\cong K_0(X,G)_0$ \cite[Proposition~4.10]{BS21}.
\end{proof}

To compute $K_0(X,G/\mu)$ from a graded presentation of $K_0(X,G)$, we use the following lemma.

\begin{lemma}\label{lem:gradedextraction}
Let $A = \bigoplus_\chi A_\chi$ be a ring graded by a finite abelian group and $I\subseteq A$ a graded ideal.  Then $(A/I)_0 \cong  A_0/(I\cap A_0)$, and if $I$ is generated by homogeneous elements $a_i$ of weights $\chi_i$ then $I\cap A_0 = \sum_i a_i\,A_{-\chi_i}$.
\end{lemma}

\begin{proof}
Both statements hold degreewise.  An element of $I$ has weight-zero component $\sum_i a_ir_i$ with $r_i$ the $(-\chi_i)$-component of the coefficient of $a_i$.
\end{proof}

\subsection{\texorpdfstring{Bundle classes and $K$-theoretic Chern classes}{Bundle classes and K-theoretic Chern classes}}
\label{subsec:kchern}

We use two kinds of generators for $K_0$-rings: classes of vector bundles, which we call \emph{bundle classes}, and \emph{$K$-theoretic Chern classes}. On smooth quasi-projective schemes, the latter come from the algebraic orientation of periodic $K$-theory $K_0[\beta,\beta^{-1}]$ in the sense of Levine--Morel \cite[Example~1.1.5]{LM07}, specialized at $\beta=1$, and the exterior-power construction below also defines them on smooth quotient stacks. We compare the two kinds of classes and then recall the gamma-filtration, which we use to compare $K_0$-rings with Chow rings.

Let $\mathcal E$ be a vector bundle of rank $r$ on a smooth quotient stack.  Taking exterior powers induces the \emph{lambda-operations} $\lambda^j([\mathcal E])=[\wedge^j\mathcal E]$ and the \emph{total lambda polynomial} $\lambda_t([\mathcal E])=\sum_{j=0}^r[\wedge^j\mathcal E]t^j$. The \emph{Whitney sum formula} $\lambda_t(z+z')=\lambda_t(z)\lambda_t(z')$ extends $\lambda_t$ to all of $K_0$ by $\lambda_t(z-z')=\lambda_t(z)/\lambda_t(z')$. The associated gamma-operations $\gamma^j$ are defined by $\gamma_t(z)=\lambda_{t/(1-t)}(z)$. For more details, see \cite[Chapter~II, Section~4]{Wei13}. 

After pullback to the flag bundle $\operatorname{Fl}(\mathcal{E})$, over which $\cE$ acquires a filtration with line-bundle quotients, one has $[\mathcal E]=\alpha_1+\cdots+\alpha_r$, where the $\alpha_i$ are the classes of these quotients.  The \emph{splitting principle} is the statement that this pullback is injective on $K_0$, so identities proved in these line-bundle classes hold on the original stack.

The \emph{$K$-theoretic Chern roots} are $\beta_i:=1-\alpha_i^{-1}$.  The \emph{$K$-theoretic Chern polynomial} is defined as
\begin{equation}\label{eq:kchern-total}
c_t^K(\mathcal E):=\sum_{j=0}^r c_j^K(\mathcal E)t^j
=\prod_{i=1}^r(1+\beta_i t),
\qquad c_0^K(\mathcal E)=1,
\end{equation}
where the coefficients $c_j^K(\mathcal E)$ are the \emph{$K$-theoretic Chern classes}. They have the property that $c_j^K(\mathcal E)=0$ for $j>r$. The following formulas express these coefficients in exterior powers of $\mathcal E^\vee$, showing that they are well-defined on $K_0$ and do not depend on the chosen splitting.

\begin{lemma}[{compare \cite[Proposition~IV.7.4]{Kar78}}]\label{lem:kchern-exterior}
For a vector bundle $\mathcal E$ of rank $r$ on a smooth quotient stack,
\[
c_t^K(\mathcal E)
=(1+t)^r\lambda_{-t/(1+t)}([\mathcal E^\vee])
=\gamma_{-t}([\mathcal E^\vee]-r).
\]
For $0\le j\le r$, the coefficients satisfy
\begin{align*}
c_j^K(\mathcal E)
&=\sum_{i=0}^j(-1)^i\binom{r-i}{j-i}
[\wedge^i\mathcal E^\vee],\\
[\wedge^j\mathcal E^\vee]
&=\sum_{i=0}^j(-1)^i\binom{r-i}{j-i}
c_i^K(\mathcal E).
\end{align*}
In particular,
\[
c_1^K(\mathcal E)=r-[\mathcal E^\vee],
\qquad
[\det\mathcal E^\vee]=\sum_{j=0}^r(-1)^j c_j^K(\mathcal E),
\qquad
c_r^K(\mathcal E)=\lambda_{-1}([\mathcal E^\vee]).
\]
\end{lemma}

\begin{proof}
In the line-bundle classes, each factor of the first expression is $(1+t)(1-\alpha_i^{-1}t/(1+t))=1+\beta_i t$. The equality with $\gamma_{-t}$ follows from its definition and $\lambda_z(r)=(1+z)^r$.  Expanding gives
\[
c_t^K(\mathcal E)=\sum_{i=0}^r
(-1)^i[\wedge^i\mathcal E^\vee]t^i(1+t)^{r-i},
\]
whose coefficient of $t^j$ is the first formula.  Since $t\mapsto -t/(1+t)$ is an involution, the first identity also gives $\lambda_t([\mathcal E^\vee])=(1+t)^r c_{-t/(1+t)}^K(\mathcal E)$, and the same coefficient calculation gives the second formula.  Taking $j=1$ and $j=r$ gives the remaining identities.
\end{proof}

\begin{theorem}[Projective bundle formula {\cite[Theorem~3.1]{Tho87}}]\label{thm:projectivebundle}
Let $\mathcal E$ be a vector bundle of rank $r\ge1$ on a smooth quotient stack $\mathscr X$, and let $\bP(\mathcal E)$ parametrize lines in $\mathcal E$. With $\ell=[\cO_{\bP(\mathcal E)}(-1)]$ and $\zeta=c_1^K(\cO_{\bP(\mathcal E)}(1))=1-\ell$, pullback induces isomorphisms
\begin{equation}\label{eq:projectivebundle}
K_0(\bP(\mathcal E))\cong\frac{K_0(\mathscr X)[\ell]}{\bigl(\sum_{i=0}^r(-1)^i[\wedge^i\mathcal E^\vee]\ell^i\bigr)}\cong\frac{K_0(\mathscr X)[\zeta]}{\bigl(\sum_{i=0}^r(-1)^ic_i^K(\mathcal E^\vee)\zeta^{r-i}\bigr)}.
\end{equation}
This is a free $K_0(\mathscr X)$-module with basis $1,\ell,\ldots,\ell^{r-1}$, or equivalently $1,\zeta,\ldots,\zeta^{r-1}$.
\end{theorem}

The top Chern class $c_r^K(\mathcal E)=\lambda_{-1}([\mathcal E^\vee])$ is the \emph{$K$-theoretic Euler class}. If $\mathcal E$ has a regular section with zero locus $i\colon \sZ\hookrightarrow\mathscr X$, its Koszul resolution gives
\[
i_*[\cO_\sZ]=\sum_{j=0}^r(-1)^j[\wedge^j\mathcal E^\vee]
=c_r^K(\mathcal E)
\quad\text{in }K_0(\mathscr X).
\]
Thus the Euler class represents the \emph{$K$-theoretic fundamental class} of $\sZ$ in the ambient stack. In particular, an effective Cartier divisor $\sD$ satisfies $[\cO_\sD]=c_1^K(\cO_{\mathscr X}(\sD))=1-[\cO_{\mathscr X}(-\sD)]$.

For a rank-two bundle, \Cref{lem:kchern-exterior} gives
\begin{equation}\label{eq:kchern-ranktwo}
\begin{aligned}
[\det\mathcal E^\vee]
&=1-c_1^K(\mathcal E)+c_2^K(\mathcal E),\\
[\mathcal E]
&=\frac{2-c_1^K(\mathcal E)}
{1-c_1^K(\mathcal E)+c_2^K(\mathcal E)},\\
c_1^K(\det\mathcal E)
&=c_1^K(\mathcal E)-c_2^K(\mathcal E).
\end{aligned}
\end{equation}
Here, the denominator is the unit $[\det\mathcal E^\vee]$, and the second equality also uses $\mathcal E\cong\mathcal E^\vee\otimes\det\mathcal E$.

For the comparison with the Chow ring, we use the $\gamma$-filtration, whose $j$-th term contains the $j$-th $K$-theoretic Chern class of every vector bundle.  Let $\sX$ be a connected smooth quotient stack, and let $\operatorname{rk}\colon K_0(\sX)\to\bZ$ be the rank map. The \emph{$\gamma$-filtration} \cite[Chapter~III, Section~1]{FL85} starts with $F_\gamma^0K_0(\sX)=K_0(\sX)$ and $F_\gamma^1K_0(\sX)=\ker(\operatorname{rk})$. For $n\ge2$, let $F_\gamma^nK_0(\sX)$ be the ideal generated by products
\[
\gamma^{i_1}(z_1)\cdots\gamma^{i_m}(z_m),
\qquad
z_j\in\ker(\operatorname{rk}),\quad
i_j\ge1,\quad i_1+\cdots+i_m\ge n.
\]
It gives a decreasing filtration with
\[
F_\gamma^aK_0(\sX)\,F_\gamma^bK_0(\sX)
\subseteq F_\gamma^{a+b}K_0(\sX)
\qquad(a,b\ge0).
\]
For a vector bundle $\mathcal E$ of rank $r$, \Cref{lem:kchern-exterior} gives
\begin{equation}\label{eq:kchern-gamma}
c_j^K(\mathcal E)
=(-1)^j\gamma^j([\mathcal E^\vee]-r)
\in F_\gamma^jK_0(\sX).
\end{equation}
Consequently, a product of Chern classes of total degree at least $n$ belongs to $F_\gamma^nK_0(\sX)$.

\subsection{Riemann--Roch for stacks}\label{subsec:riemannroch}

The Riemann--Roch theorem gives another important comparison map between the rational $K_0$-ring of a stack and the rational Chow ring of its inertia stack. To describe this map, let $\mathscr X$ be a smooth separated Deligne--Mumford stack of finite type over an algebraically closed field of characteristic zero. The inertia stack $\sI\mathscr X$ parametrizes pairs $(x,h)$, where $x$ is an object of $\mathscr X$ and $h$ is an automorphism of $x$, and the projection $p\colon\sI\mathscr X\to\mathscr X$ forgets $h$ \cite[Definition~4.2.10]{Alp26}. The universal automorphism of $p$ acts as $h$ at the pair $(x,h)$.

For a vector bundle $\mathcal E$ on $\mathscr X$, let $\zeta$ run over the eigenvalues of $h$ on $\mathcal E_x$ as $(x,h)$ varies in $\sI\mathscr X$, and let $(p^*\mathcal E)_\zeta$ denote the eigenbundle with eigenvalue $\zeta$. Since $h$ has finite order, each $\zeta$ is a root of unity in the base field. To\"en's \emph{Chern character} \cite[Definition~4.5]{Toe99} is given by
\[
\operatorname{Ch}([\mathcal E])
=\sum_\zeta\zeta\,\operatorname{ch}\bigl((p^*\mathcal E)_\zeta\bigr)
\quad\text{in }CH^\bullet(\sI\mathscr X)\otimes\overline{\bQ}.
\]
On the component of $\sI\sX$ corresponding to the identity automorphism, this agrees with the ordinary Chern character. Combining To\"en's comparison with the \'etale $K$-theory of the inertia stack \cite[Theorem~3.15]{Toe99} with the rational Chern character gives the following result. 

\begin{theorem}[Riemann--Roch for DM stacks, \cite{Toe99}]\label{thm:inertiachern}
Let $\mathscr X$ be a smooth separated Deligne--Mumford stack of finite type over an algebraically closed field of characteristic zero, with quasi-projective coarse moduli space and the resolution property. The Chern character induces an isomorphism of $\overline{\bQ}$-algebras
\begin{equation}\label{eq:inertiachern}
\operatorname{Ch}\colon K_0(\mathscr X)\otimes\overline{\bQ}
\xrightarrow{\cong}
CH^\bullet(\sI\mathscr X)\otimes\overline{\bQ}
\cong
\prod_{\mathscr Z\in\pi_0(\sI\mathscr X)}
\bigl(CH^\bullet(\mathscr Z)\otimes\overline{\bQ}\bigr).
\end{equation}
\end{theorem}

Note that the Todd class is irrelevant in our situation, since we are not computing pushforwards. 

\subsection{\texorpdfstring{The polynomials $P_k$ and $Q_k$}{The polynomials Pk and Qk}} \label{subsec:polynomial-families}

We derive recurrence relations and generating functions for $P_k$ and $Q_k$, which express symmetric powers and irreducible $\PGL_2$-representations in terms of bundle classes.

\begin{lemma}\label{lem:family}
The polynomials $P_k(X,Y)$ and $Q_k(Z)$ satisfy the following identities in $\bZ[X,Y]$ and $\bZ[Z]$, where $X,Y,Z$ are formal variables.
\begin{enumerate}[leftmargin=2.4em]
\item $P_0 = 1$, $P_1 = X$, $Q_0 = 1$, $Q_1 = Z$, and for $k\ge1$,
\[
P_{k+1} \;=\; X\,P_k - Y\,P_{k-1},
\qquad
Q_{k+1} \;=\; (Z-1)\,Q_k - Q_{k-1}.
\]
\item The generating functions are
\[
\sum_{k\ge0}P_k(X, Y)\,t^k \;=\; \frac1{1 - Xt + Yt^2},
\qquad
\sum_{k\ge0}Q_k(Z)\,t^k \;=\; \frac{1+t}{1 - (Z-1)t + t^2}.
\]
\item Under the substitution $X = \alpha+\beta$, $Y = \alpha \beta$, which embeds $\bZ[X, Y]$ as the symmetric subring of $\bZ[\alpha, \beta]$, $P_k$ becomes the complete homogeneous symmetric polynomial of degree $k$,
\begin{equation}\label{eq:Pclosed}
P_k(\alpha+\beta,\ \alpha \beta) \;=\; \frac{\alpha^{\,k+1}-\beta^{\,k+1}}{\alpha-\beta}
\;=\; \sum_{i=0}^{k}\alpha^{\,i}\beta^{\,k-i}
\qquad (k\ge0),
\end{equation}
and under the substitution $Z = \rho + 1 + \rho^{-1}$, $Q_k$ becomes the symmetric Laurent polynomial
\begin{equation}\label{eq:Qclosed}
Q_k(\rho + 1 + \rho^{-1}) \;=\;
\rho^{\,k} + \rho^{\,k-1} + \dots + \rho^{-k}
\;=\; \frac{\rho^{\,k+1}-\rho^{-k}}{\rho-1}
\qquad (k\ge0, \rho \neq 1).
\end{equation}
\end{enumerate}
\end{lemma}

\begin{proof}
The values in (1) are the cases $k\le1$ of \Cref{eq:famdef}. For the generating function of $P$, sum \Cref{eq:famdef} against $t^k$, substitute $m = k - 2j$, and apply the binomial series:
\[
\sum_{k\ge0}P_k(X, Y)\,t^k
\;=\; \sum_{j\ge0}(-Yt^2)^j
\sum_{m\ge0}\binom{m+j}{j}(Xt)^m
\;=\; \sum_{j\ge0}\frac{(-Yt^2)^j}{(1-Xt)^{\,j+1}}
\;=\; \frac1{1 - Xt + Yt^2}.
\]
The recurrence for $P$ follows by multiplying by $1-Xt+Yt^2$, while the factorization $1-(\alpha+\beta)t+\alpha \beta t^2=(1-\alpha t)(1-\beta t)$ gives \Cref{eq:Pclosed}.

The definitions in \Cref{eq:famdef} also give
\[
Q_k(Z)=P_{2k}(X,1)
\qquad\text{under}\quad X^2=Z+1.
\]
At $Y=1$, the recurrence for $P$ implies $P_{2k+2}=(X^2-2)P_{2k}-P_{2k-2}$ for $k\ge1$, hence the recurrence for $Q$. Together with $Q_0=1$ and $Q_1=Z$, this gives its generating function. Finally, \Cref{eq:Pclosed} yields
\[
Q_k(\alpha^2+1+\alpha^{-2})=P_{2k}(\alpha+\alpha^{-1},1)=\sum_{j=-k}^k \alpha^{2j},
\]
which proves \Cref{eq:Qclosed} on substituting $\rho=\alpha^2$.
\end{proof}

\section{\texorpdfstring{The $K_0$-ring of the $\mmu{g+1}$-gerbe $\sX_g$}{The K0-ring of the mu(g+1)-gerbe Xg}}
\label{sec:equivariant}

In this section, we give a presentation of $K_0(\sX_g)$ by obtaining two relations from the double-root locus and proving that they generate the kernel of $R(\GL_2)\to K_0(\sX_g)$. We prove that the proposed quotient is finite free over $\bZ$ and compute its rank, then exhibit that many distinct points over $\overline{\bQ}$ to establish reducedness. Each point comes from an automorphism of a square-free form. Restriction to stabilizers shows that every element of $\ker(A_g\to K_0(\sX_g))$ vanishes at every point of $\Spec(A_g\otimes\overline{\bQ})$, and hence is zero.

We use the binary-form quotient presentation \Cref{eq:Hg-quotient-presentation}, with $N=2g+2$ and the standard action\footnote{For even $g$, \cite[Corollary~4.6(b)]{AV04} rewrote \eqref{eq:Hg-quotient-presentation} as a $\GL_2$-quotient with the twisted action $A\cdot f=\det(A)^g f(A^{-1}x)$. We keep the untwisted action and take the quotient by $\GL_2$ itself, which gives the gerbe $\sX_g$.}
\begin{equation}\label{eq:action}
A\cdot f=f(A^{-1}x).
\end{equation}
As explained in the introduction, the central $\mmu{g+1}$ acts trivially, and the quotient map $\sX_g\to\sH_g$ is a gerbe banded by $\mmu{g+1}$.\footnote{Quotienting by the full kernel $\mmu{2g+2}$ of the action would instead give the rigidification of $\sH_g$ by removing the $\mmu{2}$ coming from the hyperelliptic involution.} \Cref{thm:isotypic} identifies
\begin{equation}\label{eq:extraction}
K_0(\sH_g) \;\cong \; K_0(\sX_g)_0 ,
\end{equation}
the weight-zero subring for the $\mmu{g+1}$-character grading, so our strategy is to first compute the full ring $K_0(\sX_g)$. With $\xi$ and $\psi$ as in \eqref{eq:representationring}, so that $R(\GL_2)\cong \bZ[\xi,\psi^{\pm1}]$, let $\theta := [\cO_{\bP^N}(1)]$ with its linearization induced by the representation $\Sym^N\cU^\vee$.  The $\mmu{g+1}$-weights are
\begin{equation}\label{eq:weights}
\operatorname{wt}(\xi) = 1,
\qquad
\operatorname{wt}(\psi) = 2,
\qquad
\operatorname{wt}(\theta) = 0
\pmod{g+1}.
\end{equation}

\subsection{The relations $F_g,G_g$}
\label{subsec:devissage}
The localization sequence for $\Delta\subset\bA^{N+1}$ gives a surjection $R(\GL_2)\to K_0(\sX_g)$ whose kernel is the image of $\mathrm G_0(\Delta,\GL_2)\to R(\GL_2)$. We obtain two elements of this kernel from the proper map
\begin{equation}\label{eq:squaremaps}
\pi_1\colon\bP^1\times\bP^{N-2}\longrightarrow\bP^N,
\qquad (\ell,u)\longmapsto\ell^2u,
\end{equation}
whose image is the projective discriminant $\Delta_1\subset\bP^N$ parametrizing binary forms of degree $N$ with a repeated root. Let $\varphi_1=[\cO_{\bP^1\times\bP^{N-2}}(1,0)]$. We compute $\pi_{1*}(1)$ and $\pi_{1*}(\varphi_1)$ by equivariant localization and pull them back to the affine complement to obtain the relations.

The maps $\bP^r\times\bP^{N-2r}\to\Delta_1$, $(q,u)\mapsto q^2u$, for $1\le r\le g+1$ form an equivariant envelope of $\Delta_1$ \cites[Proposition~4.1]{EF09}[Proposition~2.4]{AP15}. For $N=6$, Edidin--Hu used this envelope and verified computationally that, after passage to the affine complement, all its pushforwards lie in the ideal generated by the two classes from $\pi_1$ \cite[Lemma~4.2 and Section~4.1]{EH24}. We instead prove directly that $F_g,G_g$ generate the entire kernel over $\bZ$ for every $g$.

We carry out the pushforward calculation using equivariant localization as in Edidin--Hu \cite[Sections~3.1--3.2]{EH24}. For $s\ge1$, let $\bP^s = \bP(\Sym^s\cU^\vee)$ be the projective space of binary forms of degree $s$, so that the coordinate of $x_0^{s-i}x_1^i$ carries the $T$-character $\alpha^{-(s-i)}\beta^{-i}$. The fixed points of $T$ are the coordinate points $z_0, \dots, z_s$, the class $[\cO_{\bP^s}(1)]$ restricts to $\alpha^{s-j}\beta^j$ at $z_j$, and, with $\theta = [\cO_{\bP^s}(1)]$,
\begin{equation}\label{eq:pointclass}
[\cO_{z_j}]_T \;=\; \prod_{i\neq j}\bigl(1 -
\alpha^{s-i}\beta^i\,\theta^{-1}\bigr),
\qquad
\lambda_{-1}\bigl(\mathcal{N}^\vee_{z_j}\bP^s\bigr) \;=\;
\prod_{i\neq j}\Bigl(1 - \frac{\alpha^{s-i}\beta^i}{\alpha^{s-j}\beta^j}\Bigr).
\end{equation}

\begin{theorem}\label{thm:closedforms}
In $K_0(\bP^N, T)$, for every $N\ge4$,
\begin{equation}\label{eq:closedforms}
\begin{aligned}
\pi_{1*}(1) \;&=\;
1 + \xi\psi\,P_{N-3}\,\theta^{-1} - \psi^2P_{2N-4}\,\theta^{-2},\\
\pi_{1*}(\varphi_1) \;&=\;
\xi + \xi\psi^2P_{N-4}\,\theta^{-1} - \psi^3P_{2N-5}\,\theta^{-2}.
\end{aligned}
\end{equation}
Both classes are symmetric in $\alpha, \beta$ and hence lie in $K_0(\bP^N, \GL_2)$.
\end{theorem}

\begin{proof}
Since $K_0(\bP^N,T)$ is free over $R(T)$ by \Cref{thm:projectivebundle}, we may verify the identities at the fixed points by \Cref{prop:toruslocalization}. The fixed points $z_{i,j}=(z_i,z_j)$ of $\bP^1\times\bP^{N-2}$ map to $z_{2i+j}\in\bP^N$, so the pushforward formula of Edidin--Hu gives
\[
(\pi_1)_*(\varphi_1^k)\big|_{z_m}
\;=\;
\sum_{2i+j = m}
\bigl(\alpha^{1-i}\beta^i\bigr)^k\,
\frac{\lambda_{-1}\bigl(\mathcal{N}^\vee_{z_m}\bP^N\bigr)}
{\lambda_{-1}\bigl(\mathcal{N}^\vee_{z_{i,j}}(\bP^1\times\bP^{N-2})\bigr)} ,
\]
where $i\in\{0,1\}$ and $0\le j\le N-2$. With $q=\beta/\alpha$, cancelling the common factors in \eqref{eq:pointclass} gives
\[
(\pi_1)_*(\varphi_1^k)\big|_{z_m}
\;=\;
\alpha^k\,\frac{(1-q^{\,N-1-m})(1-q^{\,N-m})}{1-q}
\;+\;
(\alpha q)^k\,\frac{(1-q^{-m})(1-q^{\,1-m})}{1-q^{-1}} ,
\]
where the first summand is nonzero for $m\le N-2$ and the second for $m\ge2$.  Substituting $\theta|_{z_m}=\alpha^Nq^m$, $\xi=\alpha(1+q)$, $\psi=\alpha^2q$ and $P_s=\alpha^s(1-q^{s+1})/(1-q)$ into the right-hand sides of \Cref{eq:closedforms} gives
\[
\frac{\alpha^k}{1-q}\Bigl[
(1-q^{N-1-m})(1-q^{N-m})
-q^{k+1}(1-q^{-m})(1-q^{1-m})\Bigr]
\qquad(k=0,1).
\]
The identity $(1-q)/(1-q^{-1})=-q$ shows that this expression equals $(\pi_1)_*(\varphi_1^k)\big|_{z_m}$ for every $0\le m\le N$. Since $\cO_{\bP^N}(1)$ is $\GL_2$-linearized and the coefficients in \Cref{eq:closedforms} lie in $R(\GL_2)$, both descend to classes in $K_0(\bP^N,\GL_2)$.
\end{proof}

We now pull back the two classes to $\bA^{N+1}\smallsetminus\Delta$, as in Edidin--Hu \cite[Section~4]{EH24}. For our standard $\GL_2$-action, pulling back sends $\theta$ to $1$, while Edidin--Hu obtained a factor of $\psi^2$ because the $\GL_2/\mmu{3} \cong \GL_2$-action is twisted by $(\det\cU)^{\otimes2}$.

\begin{lemma}\label{lem:torsor}
The projection $\bA^{N+1}\smallsetminus\{0\}\to\bP^N$ restricts to a $\bG_m$-torsor $\bA^{N+1}\smallsetminus\Delta\to\bP^N\smallsetminus\Delta_1$, associated to the line bundle $\cO_{\bP^N}(-1)$ with the linearization induced by $\Sym^N\cU^\vee$, and, for $G$ either $\GL_2$ or its diagonal torus $T$,
\[
K_0\bigl(\bA^{N+1}\smallsetminus\Delta, G\bigr) \;\cong \;
K_0\bigl(\bP^N\smallsetminus\Delta_1, G\bigr)\big/(1 - \theta) .
\]
\end{lemma}

\begin{proof}
The scheme $\bA^{N+1}\smallsetminus\Delta$ is the complement of the zero section of $\cO_{\bP^N}(-1)$ on $\bP^N\smallsetminus\Delta_1$. The localization sequence and homotopy invariance give the claimed quotient, since the zero-section pushforward is multiplication by $1-[\cO_{\bP^N}(1)]=1-\theta$.
\end{proof}

We denote the two relations obtained under $\theta\mapsto1$ by
\[
F_g=1+\xi\psi P_{N-3}-\psi^2P_{2N-4},\qquad
G_g=\xi+\xi\psi^2P_{N-4}-\psi^3P_{2N-5},
\]
and let
\[
A_g:=\bZ[\xi,\psi^{\pm1}]/(F_g,G_g),
\]
which is our proposed presentation for $K_0(\sX_g)$.

\begin{proposition}\label{prop:eqsurj}
The classes $F_g$ and $G_g$ are homogeneous for the $\mmu{g+1}$-character grading, of weights $0$ and $1$, and restriction of representations induces a surjection of $\mu_{g+1}^\vee$-graded rings
\[
A_g \twoheadrightarrow
K_0\bigl(\bA^{N+1}\smallsetminus\Delta, \GL_2\bigr)\cong K_0(\sX_g).
\]
The map sends $\xi$ to $[\cU]$ and $\psi$ to $[\det\cU]$.
\end{proposition}

\begin{proof}
By homotopy invariance and the localization sequence \Cref{eq:localization} for $\Delta\subset\bA^{N+1}$, restriction $R(\GL_2)\to K_0(\sX_g)$ is surjective. The two pushforwards in \Cref{thm:closedforms} vanish on $\bP^N\smallsetminus\Delta_1$, since $\pi_1$ factors through $\Delta_1$.  Pullback to the affine complement sends $\theta$ to $1$ by \Cref{lem:torsor}, so $F_g$ and $G_g$ vanish in $K_0(\sX_g)$ and the map factors through $A_g$. Since $P_k(\xi,\psi)$ has weight $k$, the polynomials $F_g,G_g$ are homogeneous of weights $0,1$, respectively. Restriction of representations preserves the $\mmu{g+1}$-character grading.
\end{proof}

\subsection{The kernel of $R(\GL_2) \twoheadrightarrow K_0(\sX_g)$}
\label{subsec:masterspec}

To prove that $A_g\to K_0(\sX_g)$ is injective, we first study $A_g$ after extending scalars from $R(\GL_2)$ to $R(T)$. Change-of-groups on equivariant $K_0$ gives a restriction map $K_0(\sX_g)\to K_0(\bA^{N+1}\smallsetminus\Delta,T)$. On the level of representation rings, it gives the inclusion
\[
\begin{gathered}
R(\GL_2)\cong\bZ[\xi,\psi^{\pm1}]\cong R(T)^{\mathfrak S_2}
\hookrightarrow R(T),\\
R(T)\cong\bZ[\alpha^{\pm1},\beta^{\pm1}],
\qquad \xi=\alpha+\beta,\quad\psi=\alpha \beta,
\end{gathered}
\]
where $\mathfrak S_2$ acts by interchanging the character classes $\alpha$ and $\beta$. The relations $F_g,G_g$ vanish after restriction to $T$, giving a surjection $R(T)/(F_g,G_g)\twoheadrightarrow K_0(\bA^{N+1}\smallsetminus\Delta,T)$.

Since
\[
R(T)\cong R(\GL_2)[t]/(t^2-\xi t+\psi)\qquad \text{by } t\longmapsto \alpha,
\]
the ring $R(T)$ is free over $R(\GL_2)$ with basis $1,\alpha$. Hence, $R(T)/(F_g,G_g)\cong A_g\otimes_{R(\GL_2)}R(T)$ is a free $A_g$-module with the same basis, so finite freeness of $R(T)/(F_g,G_g)$ over $\bZ$ implies finite freeness of $A_g$, with half the rank. The two factorizations below reduce this computation to the indices of four sublattices of $\bZ^2$.

\begin{lemma}\label{lem:branch}
In $R(T) \cong \bZ[\alpha^{\pm1}, \beta^{\pm1}]$,
\begin{equation}\label{eq:branch}
G_g - \alpha\,F_g \;=\; \beta\,(1 - \alpha^N)(1 - \alpha^{N-1}\beta),
\qquad
G_g - \beta\,F_g \;=\; \alpha\,(1 - \beta^N)(1 - \alpha \beta^{N-1}).
\end{equation}
\end{lemma}

\begin{proof}
Substituting $\xi=\alpha+\beta$ and $\psi=\alpha \beta$ into $F_g$ and $G_g$ and expanding with $(\alpha-\beta)P_k=\alpha^{k+1}-\beta^{k+1}$ gives
\[
(\alpha-\beta)(G_g-\alpha F_g)=(\alpha-\beta)\beta(1-\alpha^N)(1-\alpha^{N-1}\beta).
\]
After cancelling $\alpha-\beta$, we obtain the first identity of \Cref{eq:branch}.  Exchanging $\alpha$ and $\beta$ gives the second identity.
\end{proof}

\Needspace{6\baselineskip}
\begin{proposition}\label{prop:masterflat}
The algebra $A_g=\bZ[\xi,\psi^{\pm1}]/(F_g,G_g)$ is a free abelian group of rank $N(2N-3)=(g+1)(8g+2)$.
\end{proposition}

\begin{proof}
Let $H=G_g-\alpha F_g$ and $J=G_g-\beta F_g$. We first compute the rank of $R(T)/(H,J)$ using the factorizations in \Cref{eq:branch}, then recover the rank of $R(T)/(F_g,G_g)$ from the identity $J-H=(\alpha-\beta)F_g$.

For linearly independent vectors $(i,j),(r,s)\in\bZ^2$, the quotient
\[
R(T)/(1-\alpha^i\beta^j,\,1-\alpha^r\beta^s)
\]
has a $\bZ$-basis consisting of one monomial for each coset of $\bZ(i,j)+\bZ(r,s)$ in $\bZ^2$. Its rank is $[\bZ^2:\bZ(i,j)+\bZ(r,s)]=|is-jr|$, giving the following four ranks.
\[
\begin{array}{c|c}
I & \operatorname{rank}_{\bZ}(R(T)/I)\\ \hline
(1-\alpha^N,\,1-\beta^N) & N^2\\
(1-\alpha^N,\,1-\alpha \beta^{N-1}) & N(N-1)\\
(1-\alpha^{N-1}\beta,\,1-\beta^N) & N(N-1)\\
(1-\alpha^{N-1}\beta,\,1-\alpha \beta^{N-1}) & N(N-2).
\end{array}
\]
The two generators of $I$ in each row are coprime in $R(T)$. For $f,h,u\in R(T)$ with $f$ and $u$ having no common irreducible factor, multiplication by $u$ gives an exact sequence
\[
0\longrightarrow R(T)/(f,h)\xrightarrow{\ \cdot u\ }R(T)/(f,hu)
\longrightarrow R(T)/(f,u)\longrightarrow0.
\]
Applying this sequence successively to the factors of $H$ and $J$ gives a filtration of $R(T)/(H,J)$ with the four quotients $R(T)/I$ in the above table. Since all four are finite free over $\bZ$, we conclude that $R(T)/(H,J)$ is a free abelian group of rank $4N(N-1)$.

To recover $R(T)/(F_g,G_g)$, we use the fact that $J-H=(\alpha-\beta)F_g$. The element $\alpha-\beta$ is prime in $R(T)$ and does not divide $H$, since substituting $\beta=\alpha$ gives $H=\alpha(1-\alpha^N)^2\ne0$. Applying the preceding exact sequence with $f=H$, $h=F_g$, and $u=\alpha-\beta$ gives
\[
0\longrightarrow R(T)/(F_g,G_g)
\xrightarrow{\ \cdot(\alpha-\beta)\ }R(T)/(H,J)
\longrightarrow R(T)/(H,\alpha-\beta)\longrightarrow0.
\]
The last quotient is isomorphic to $\bZ[\alpha^{\pm1}]/((1-\alpha^N)^2)$, which is free of rank $2N$. Thus, $R(T)/(F_g,G_g)$ is finite free of rank $4N(N-1)-2N=2N(2N-3)$. As a $\bZ$-module, it is isomorphic to $A_g\oplus A_g$, so $A_g$ is finite free of rank $N(2N-3)$.
\end{proof}

We now describe the closed points of the finite scheme $\Spec(A_g\otimes\overline{\bQ})$. Since $R(\GL_2) \cong R(T)^{\mathfrak S_2}$, these points can be parametrized by unordered pairs $\{\alpha,\beta\}$, with $\alpha,\beta\in\overline{\bQ}^{\times}$, satisfying $F_g(\alpha+\beta,\alpha \beta)=G_g(\alpha+\beta,\alpha \beta)=0$. Exhibiting $N(2N-3)$ distinct points will prove reducedness, since this equals the dimension of $A_g\otimes\overline{\bQ}$ by \Cref{prop:masterflat}.

\begin{proposition}\label{prop:masterspec}
The scheme $\Spec(A_g\otimes\overline{\bQ})$ consists of $(g+1)(8g+2)$ reduced points, parametrized by the unordered pairs $\{\alpha,\beta\}$ satisfying one of the following conditions, where it is understood that $(\alpha+\beta,\alpha \beta) = (\xi,\psi)$:
\begin{enumerate}[leftmargin=2.4em]
\item $\alpha^N = \beta^N = 1$, giving $\binom{N}{2}$ points with $\alpha\neq \beta$ and $N$ points with $\alpha = \beta$.
\item $\alpha^N = 1$ and $\alpha \beta^{N-1} = 1$ with $\alpha\neq \beta$, or the same with $\alpha$ and $\beta$ exchanged, giving $N(N-2)$ points.
\item $\alpha^{N-1}\beta = \alpha \beta^{N-1} = 1$ with $\alpha\neq \beta$, giving $N(N-3)/2$ points.
\end{enumerate}
\end{proposition}

\begin{proof}
We first verify that the listed points are distinct elements of $\Spec(A_g\otimes\overline{\bQ})$. For $\alpha\ne \beta$, the identities in \Cref{eq:branch} show that each pair is a common zero of $F_g$ and $G_g$. For $\alpha=\beta$ with $\alpha^N=1$, the equality $P_k(2\alpha,\alpha^2)=(k+1)\alpha^k$ gives the same conclusion by substitution. The coordinates $(\alpha+\beta,\alpha \beta)$ determine the unordered pair $\{\alpha,\beta\}$. No point can satisfy more than one condition: satisfying two of the listed conditions would impose both $\alpha^N=1$ and $\alpha^{N-1}\beta=1$, or the same equations with $\alpha$ and $\beta$ exchanged, forcing $\alpha=\beta$, which is excluded in (2) and (3).

The equations $\alpha^N=\beta^N=1$, $\alpha^N=\alpha \beta^{N-1}=1$, and $\alpha^{N-1}\beta=\alpha \beta^{N-1}=1$ have $N^2$, $N(N-1)$, and $N(N-2)$ ordered solutions, respectively. Each has $N$ solutions with $\alpha=\beta$. Exchanging $\alpha$ and $\beta$ preserves the first and third systems, whereas the second system and its exchanged version meet only where $\alpha=\beta$. The numbers of unordered pairs are therefore $N(N+1)/2$, $N(N-2)$ and $N(N-3)/2$, whose sum is $N(2N-3)$. By \Cref{prop:masterflat}, this is also the dimension of $A_g\otimes\overline{\bQ}$. Hence, these points exhaust this spectrum and are all reduced.
\end{proof}

Finite freeness over $\bZ$ and reducedness over $\overline{\bQ}$ imply that elements of $A_g$ are determined by their values at these points, as follows.

\begin{lemma}\label{lem:generic-evaluation}
Let $A$ be a commutative $\bZ$-algebra that is torsion-free as a $\bZ$-module. If $A\otimes_{\bZ}\overline{\bQ}$ is reduced and finite-dimensional over $\overline{\bQ}$, then evaluation induces an injective ring homomorphism
\[
A\longrightarrow
\prod_{x\in\Spec(A\otimes_{\bZ}\overline{\bQ})}\overline{\bQ},
\qquad f\longmapsto\bigl(f(x)\bigr)_x.
\]
\end{lemma}

\begin{proof}
Torsion-freeness implies that $A\to A\otimes_{\bZ}\overline{\bQ}$ is injective. Since a reduced finite-dimensional algebra over an algebraically closed field is a product of copies of that field, evaluation identifies $A\otimes_{\bZ}\overline{\bQ}$ with the desired product.
\end{proof}

To apply \Cref{lem:generic-evaluation} to $\ker(A_g\to K_0(\sX_g))$, we will show that each point of $\Spec(A_g\otimes\overline{\bQ})$ is represented by the eigenvalues of an automorphism fixing a square-free binary form. A virtual representation whose class vanishes in $K_0(\sX_g)$ then restricts to zero in the representation ring of the stabilizer of such a form, so its character vanishes at that automorphism.

\Needspace{6\baselineskip}
\begin{lemma}\label{lem:realize}
Every point of $\Spec(A_g\otimes\overline{\bQ})$ is parametrized by $(\xi,\psi)=(\alpha+\beta,\alpha \beta)$ for some $\alpha,\beta\in\overline{\bQ}^\times$ such that $\operatorname{diag}(\alpha,\beta)$ fixes a square-free binary form of degree $N$ over $\overline{\bQ}$.
\end{lemma}

\begin{proof}
For the points of types (1), (2), and (3) of \Cref{prop:masterspec}, respectively, we take
\[
x_0^N+x_1^N,\qquad
x_0^N+x_0x_1^{N-1},\qquad
x_0^{N-1}x_1+x_0x_1^{N-1}.
\]
All three forms are square-free over $\overline{\bQ}$. The equations $\alpha^N=\beta^N=1$, $\alpha^N=\alpha \beta^{N-1}=1$, and $\alpha^{N-1}\beta=\alpha \beta^{N-1}=1$, respectively, ensure that both monomials are fixed by $\operatorname{diag}(\alpha,\beta)$. 
\end{proof}

\begin{lemma}\label{lem:orbitvanish}
Let $t_0 = \operatorname{diag}(\alpha_0, \beta_0)$ be a $\overline{\bQ}$-point of  $T$ fixing a square-free binary form $f_0$ of degree $N$ over $\overline{\bQ}$.  Then every element of the kernel of $R(T)\to K_0(\bA^{N+1}\smallsetminus\Delta, T)$ vanishes at $(\alpha_0, \beta_0)$.
\end{lemma}

\begin{proof}
Since the base field $\Bbbk$ has characteristic zero, $\overline{\bQ}$ embeds in $\bar{\Bbbk}$, so $t_0$ and $f_0$ are defined over $\bar{\Bbbk}$. Equivariant vector bundles on the orbit of $f_0$ correspond to representations of the diagonalizable group $\operatorname{Stab}_T(f_0)$. Base change to $\bar{\Bbbk}$ followed by pullback to this orbit gives ring maps
\[
R(T)\longrightarrow
K_0(\bA^{N+1}\smallsetminus\Delta,T)
\longrightarrow R(\operatorname{Stab}_T(f_0)),
\]
whose composite is restriction of characters. Since $t_0\in\operatorname{Stab}_T(f_0)$, evaluation at $t_0$ factors through $R(\operatorname{Stab}_T(f_0))$ and therefore vanishes on $\ker\bigl(R(T)\to K_0(\bA^{N+1}\smallsetminus\Delta,T)\bigr)$.
\end{proof}

\begin{theorem}\label{thm:eqpres}
Restriction of representations induces an isomorphism of $\mmu{g+1}^\vee$-graded rings
\[
A_g=\bZ[\xi,\psi^{\pm1}]/(F_g,G_g)
\;\xrightarrow{\ \sim\ }\; K_0(\sX_g),
\qquad
\xi\longmapsto[\cU],\quad
\psi\longmapsto[\det\cU].
\]
\end{theorem}

\begin{proof}
By \Cref{prop:eqsurj}, it remains to prove injectivity. If $h\in R(\GL_2)$ maps to zero in $K_0(\sX_g)$, \Cref{lem:realize,lem:orbitvanish} show that $h$ vanishes at every point of $\Spec(A_g\otimes\overline{\bQ})$. \Cref{prop:masterflat,prop:masterspec} establish the hypotheses of \Cref{lem:generic-evaluation}, so $h\in(F_g,G_g)$.
\end{proof}

\begin{corollary}\label{prop:gerbefinite}
The ring $K_0(\sX_g)$ is a free abelian group of rank $(g+1)(8g+2)$.
\end{corollary}

\begin{proof}
By \Cref{thm:eqpres}, we have $K_0(\sX_g)\cong A_g$, which is finite free of the stated rank by \Cref{prop:masterflat}.
\end{proof}

The discriminant gives a further relation that will allow us to express inverses of determinant classes as polynomials in the Chern-class presentations.

\begin{proposition}\label{prop:discriminantclass}
The discriminant divisor has fundamental class $[\cO_\Delta]=1-\psi^{N(N-1)}$ in $K_0(\bA^{N+1},\GL_2)$, and $\psi^{N(N-1)}=1$ in $K_0(\sX_g)$.
\end{proposition}

\begin{proof}
For the action $A\cdot f=f(A^{-1}x)$, the discriminant satisfies
\[
\disc(f\circ A^{-1})=\det(A)^{-N(N-1)}\disc(f)
\]
by the discriminant calculation in \cite[proof of Theorem~5.1]{AV04}.  Thus, it is an equivariant section of $\det^{-N(N-1)}$ with zero divisor $\Delta$. The divisor sequence gives $[\cO_\Delta]=1-\psi^{N(N-1)}$, and restriction to the square-free locus gives $\psi^{N(N-1)}=1$.
\end{proof}

\section{\texorpdfstring{The $K_0$-ring of $\sH_g$}{The K0-ring of Hg}}\label{sec:khg}

Our calculation of $K_0(\sX_g)$ in \Cref{thm:eqpres} reduces the computation of $K_0(\sH_g)$ to identifying the weight-zero subring of $A_g$.  Since $F_g$ has $\mmu{g+1}$-weight zero and $G_g$ has weight one, \Cref{lem:gradedextraction} gives
\begin{equation}\label{eq:parityplan}
K_0(\sH_g)\cong(A_g)_0
\cong \frac{\bZ[\xi,\psi^{\pm1}]_0}
{F_g\,\bZ[\xi,\psi^{\pm1}]_0+
G_g\,\bZ[\xi,\psi^{\pm1}]_{-1}}.
\end{equation}
Before computing the generators and relations in \Cref{eq:parityplan}, we determine the ranks of the $\mmu{g+1}$-character summands.

\begin{proposition}\label{prop:characterranks}
For every $g\ge2$ and every $\mmu{g+1}$-weight $j\in\bZ/(g+1)$, the summand $K_0(\sX_g)_j$ is a free abelian group of rank
\[
\operatorname{rank}_{\bZ}K_0(\sX_g)_j=
\begin{cases}
8g+2,&g\text{ even},\\
8g+4,&g\text{ odd and }j\text{ even},\\
8g,&g\text{ odd and }j\text{ odd}.
\end{cases}
\]
In particular, $K_0(\sH_g)\cong K_0(\sX_g)_0$ is free of rank $8g+2$ for even $g$ and $8g+4$ for odd $g$. The generic fiber of $\operatorname{Spec}K_0(\sH_g)$ is reduced.
\end{proposition}

\begin{proof}
Each direct summand $(A_g)_j\cong K_0(\sX_g)_j$ is finite free over $\bZ$ by \Cref{prop:masterflat}, and we denote its rank by $r_j$. Multiplication by $\psi$ gives $(A_g)_j\cong(A_g)_{j+2}$, so $r_j=r_{j+2}$. If $g$ is even, addition by two cycles through all weights in $\bZ/(g+1)$, and the total rank $(g+1)(8g+2)$ gives $r_j=8g+2$ for every $j$.

For odd $g$, the ranks are constant within each parity, and we compare them using multiplication by $\xi$. The ring $A_g$ is reduced, since it embeds in the reduced ring $A_g\otimes\overline{\bQ}$ of \Cref{prop:masterspec}. Every odd-weight element $h$ is divisible by $\xi$, because $\psi$ has even weight. If $h=\xi u$ and $\xi h=0$, then $h^2=u\xi h=0$, so $h=0$. Thus, multiplication by $\xi$ is injective on each odd-weight summand, and for every even $j$ we have an exact sequence
\[
0\longrightarrow(A_g)_{j-1}\xrightarrow{\ \cdot\xi\ }(A_g)_j
\longrightarrow\bigl(A_g/(\xi)\bigr)_j\longrightarrow0.
\]
Since $A_g/(\xi)\cong\bZ[\psi]/(\psi^N-1)$, its basis $1,\psi,\ldots,\psi^{N-1}$ contains four elements of each even weight modulo $g+1$. Hence, $r_j-r_{j-1}=4$ for even $j$. Together with the total rank $(g+1)(8g+2)$, this gives rank $8g+4$ in even weights and $8g$ in odd weights.

Finally, $K_0(\sH_g)\cong(A_g)_0$, and $(A_g)_0\otimes\overline{\bQ}$ is a subring of the reduced ring $A_g\otimes\overline{\bQ}$ from \Cref{prop:masterspec}. This proves the assertions about $K_0(\sH_g)$.
\end{proof}

Di Lorenzo and Pirisi \cite[Section~8]{DLP21} showed that the universal conic $\sC/\iota\to\sH_g$ has nonzero Brauer class when $g$ is odd, so it cannot be a projective bundle. Our rank computations in \Cref{prop:characterranks} give another proof to this fact, since they imply that $K_0(\sC/\iota)$ has strictly smaller rank than the projective bundle formula would predict.

\begin{corollary}\label{cor:universalconic}
For odd $g\ge3$, the universal conic $\sC/\iota\to\sH_g$ is not isomorphic over $\sH_g$ to $\bP(\mathcal F)$ for any rank-two vector bundle $\mathcal F$ on $\sH_g$.
\end{corollary}

\begin{proof}
The pullback of $\sC/\iota$ to $\sX_g$ is $\bP(\cU)$, and the tautological line bundle $\cO_{\bP(\cU)}(-1)$ has $\mmu{g+1}$-weight one. Taking the weight-zero part of the projective bundle formula and descending along the $\mmu{g+1}$-gerbe gives an isomorphism of abelian groups
\[
K_0(\sC/\iota)\cong K_0(\sX_g)_0\oplus K_0(\sX_g)_{-1}
\]
by \Cref{thm:projectivebundle,thm:isotypic}. Thus, $K_0(\sC/\iota)$ is free of rank $(8g+4)+8g=16g+4$ by \Cref{prop:characterranks}, whereas $K_0(\bP(\mathcal F))$ would have rank $2(8g+4)=16g+8$ by the projective bundle formula.
\end{proof}

The bundles $\cV$ for even $g$ and $\cL,\cE$ for odd $g$ arise from representations of $\GL_2/\mmu{g+1}$, or equivalently from representations of $\GL_2$ on which $\mmu{g+1}$ acts trivially. Taking the weight-zero subring therefore gives generators $[\cV]$ and $[\det\cV]^{\pm1}$ in even genus, and $[\cL]^{\pm1}$ and $[\cE]$ in odd genus. The Chern classes $c_1(\cV),c_2(\cV)$ are the generators in the Chow presentation of Edidin--Fulghesu \cite[Section~5.1]{EF09}, while $c_1(\cL),c_2(\cE),c_3(\cE)$ are those in the presentation of Di Lorenzo \cite[Section~6]{DL21}.

\subsection{Even genus}\label{sec:even}

Throughout this subsection, let $g$ be even, and recall the rank-two bundle $\cV$ on $\sH_g$ induced by the standard representation under $\GL_2/\mmu{g+1}\cong\GL_2$. Let $\nu=[\cV]$ and $\lambda=[\det\cV]$. Since $\cV$ pulls back to $\cU\otimes(\det\cU)^{g/2}$ on $\sX_g$ \cite[Proposition~4.4(b)]{AV04}, we use the same notation for the corresponding elements of $R(\GL_2)$:
\[
\nu := \xi\psi^{g/2},\qquad \lambda := \psi^{g+1}.
\]
We express $F_g$ and $\psi^{g/2}G_g$, both of weight zero, as polynomials in $\nu$ and $\lambda^{\pm1}$ using $P_k(\nu,\lambda)=[\Sym^k\cV]$.

\begin{proposition}\label{prop:evenlattice}
The weight-zero subring of $\bZ[\xi, \psi^{\pm1}]$ is $\bZ[\nu, \lambda^{\pm1}]$, and the weight-$(-1)$ module is $\psi^{g/2}\,\bZ[\nu, \lambda^{\pm1}]$. Furthermore, $F_g$ and $\psi^{g/2}G_g$ satisfy
\[
\begin{aligned}
F_g=F^{\mathrm{even}}_g
&:=1+\nu\lambda^{1-g}P_{2g-1}(\nu,\lambda)-\lambda^{2-2g}P_{4g}(\nu,\lambda),\\
\psi^{g/2}G_g=G^{\mathrm{even}}_g
&:=\nu+\nu\lambda^{2-g}P_{2g-2}(\nu,\lambda)-\lambda^{3-2g}P_{4g-1}(\nu,\lambda).
\end{aligned}
\]
\end{proposition}

\begin{proof}
A monomial $\xi^i\psi^j$ has weight $i+2j$ modulo $g+1$. Since $g+1$ is odd, weight zero is equivalent to $j-ig/2$ being divisible by $g+1$.  Thus $\xi^i\psi^j=\nu^i\lambda^{(j-ig/2)/(g+1)}$, which proves that the weight-zero subring is $\bZ[\nu,\lambda^{\pm1}]$. Multiplication by the unit $\psi^{g/2}$, of weight $-1$, identifies this subring with the weight-$(-1)$ module.

Taking symmetric powers of $\cU\otimes(\det\cU)^{g/2}$ gives
\[
P_k(\nu,\lambda)=\psi^{kg/2}P_k(\xi,\psi).
\]
Substituting this identity together with $\nu=\xi\psi^{g/2}$ and $\lambda=\psi^{g+1}$ into the formulas for $F_g^{\mathrm{even}}$ and $G_g^{\mathrm{even}}$ gives $F_g^{\mathrm{even}}=F_g$ and $G_g^{\mathrm{even}}=\psi^{g/2}G_g$.
\end{proof}

\begin{theorem}\label{thm:even}
Let $g\ge2$ be even.  Over a field of characteristic zero,
\[
K_0(\sH_g) \;\cong \; \bZ[\nu, \lambda^{\pm1}]\big/(F^{\mathrm{even}}_g, G^{\mathrm{even}}_g),
\]
a free abelian group of rank $8g+2$, with $\nu = [\cV]$ and $\lambda = [\det\cV]$.
\end{theorem}

\begin{proof}
The presentation follows from \eqref{eq:parityplan} and \Cref{prop:evenlattice}, while \Cref{prop:characterranks} gives finite freeness and the rank.
\end{proof}

The inclusion $(A_g)_0\subset A_g$ also determines the points of $\Spec(K_0(\sH_g)\otimes\overline{\bQ})$. Indeed, $(A_g)_0\otimes\overline{\bQ}$ is the invariant subring for the scalar $\mmu{g+1}$-action on the finite reduced algebra $A_g\otimes\overline{\bQ}$, so its spectrum is the set of orbits of this action. Restricting the characters $\nu=\xi\psi^{g/2}$ and $\lambda=\psi^{g+1}$ sends the pair $\{\alpha,\beta\}$ as in \Cref{prop:masterspec} to
\[
\{\alpha(\alpha \beta)^{g/2},\ \beta(\alpha \beta)^{g/2}\},
\]
whose sum and product are $\nu$ and $\lambda$. In the statement below, $\alpha,\beta$ denote these eigenvalues on $\cV$.

\begin{proposition}\label{prop:evenspectrum}
For even $g$, the scheme $\Spec(K_0(\sH_g)\otimes\overline{\bQ})$ consists of $8g+2$ reduced points. These points have coordinates $(\nu,\lambda)=(\alpha+\beta,\alpha \beta)$ for the following unordered pairs $\{\alpha,\beta\}$:
\begin{enumerate}[leftmargin=2.2em]
\item $\{\zeta, \zeta^{-1}\}$ with $\zeta\in\mmu{2g}\cup\mmu{2g+2}$ taken up to replacing $\zeta$ by $\zeta^{-1}$, giving $2g+1$ points with $\lambda = 1$, among them $(2, 1)$ and $(-2, 1)$.
\item $\{\zeta, -\zeta^{-1}\}$ with $\zeta^{2g+2}=1$ or $\zeta^{2g} = -1$, taken up to the substitution $\zeta\mapsto-\zeta^{-1}$, giving $2g+1$ points with $\lambda = -1$.
\item $\{\zeta \beta, \beta\}$ with $\zeta\in\mmu{2g+1}\smallsetminus\{1\}$ and $\beta^2\in\{\zeta^{g-1},\zeta^g\}$, modulo exchange of the two entries, giving $4g$ points with $\lambda\in\{\zeta^{\,g}, \zeta^{\,g+1}\}\subset \mmu{2g+1}\smallsetminus\{1\}$.
\end{enumerate}
The elimination ideal $((F_g^{\mathrm{even}},G_g^{\mathrm{even}})\otimes\bQ)\cap\bQ[\lambda^{\pm1}]$ is generated by $(\lambda^2-1)(1+\lambda+\dots+\lambda^{2g})$. Moreover, $(\lambda^2-1)(1+\lambda+\dots+\lambda^{2g})\in (F_g^{\mathrm{even}},G_g^{\mathrm{even}})\subset\bZ[\nu,\lambda^{\pm1}]$.
\end{proposition}

\begin{proof}
For a pair $\{\alpha,\beta\}$ in \Cref{prop:masterspec}, let $\rho=\alpha/\beta$. After choosing the ordering in its second case, the three conditions become
\begin{equation}\label{eq:restriction-conditions}
\begin{array}{c|c|c}
&\text{condition on }\rho&\text{condition on }\beta\\ \hline
(1)&\rho^N=1&\beta^N=1\\
(2)&\rho^{N-1}=1,\ \rho\ne1&\beta^N=\rho^{-1}\\
(3)&\rho^{N-2}=1,\ \rho\ne1&\beta^N=\rho^{-1}.
\end{array}
\end{equation}
The eigenvalues $\alpha(\alpha \beta)^{g/2}$ and $\beta(\alpha \beta)^{g/2}$ have ratio $\rho$. Moreover,
\[
\lambda=\rho^{g+1}\beta^N,
\qquad
\bigl(\beta(\alpha \beta)^{g/2}\bigr)^2=\rho^g \beta^N.
\]
The first and third rows give $\lambda=\pm1$. In the first row, both eigenvalues on $\cV$ are $(2g+2)$-th roots of unity, whereas in the third, each eigenvalue $\zeta$ satisfies $\zeta^{2g}=\lambda$. Together, these give cases (1) and (2) of the statement. In the second row, $\bigl(\beta(\alpha \beta)^{g/2}\bigr)^2=\rho^{g-1}$, and exchanging the eigenvalues gives the alternative equation $\bigl(\beta(\alpha\beta)^{g/2}\bigr)^2=\rho^g$, after replacing $\rho$ by $\rho^{-1}$. For fixed $\rho$, varying $\beta$ among the solutions of its equation gives both square roots, so all the listed pairs occur.

For $\lambda=1$, the union $\mmu{2g}\cup\mmu{2g+2}$ has $4g$ elements, of which only $1$ and $-1$ equal their own inverses, and hence gives $2g+1$ pairs. For $\lambda=-1$, the sets defined by $\zeta^{2g+2}=1$ and $\zeta^{2g}=-1$ are disjoint and have total size $4g+2$. The involution $\zeta\mapsto-\zeta^{-1}$ acts freely on their union, again giving $2g+1$ pairs. Finally, each of the $g$ unordered pairs $\{\zeta,\zeta^{-1}\}$ in $\mmu{2g+1}\smallsetminus\{1\}$ gives four points, since $\zeta^{g-1}$ and $\zeta^g$ are distinct and nonzero, and each has two square roots. Their values of $\lambda$ differ from $\pm1$, so the three cases are disjoint.

The values of $\lambda$ are $\pm1$ and $\mmu{2g+1}\smallsetminus\{1\}$, because raising to the $g$-th power permutes $\mmu{2g+1}$. Since $K_0(\sH_g)\otimes\bQ$ is reduced, a Laurent polynomial in $\bQ[\lambda^{\pm1}]$ vanishes in this ring if and only if it is divisible by $(\lambda^2-1)(1+\cdots+\lambda^{2g})$. Torsion-freeness of $K_0(\sH_g)$ then gives $(\lambda^2-1)(1+\cdots+\lambda^{2g})\in(F_g^{\mathrm{even}},G_g^{\mathrm{even}})$ over $\bZ$.
\end{proof}

At $g=2$, we recover Edidin--Hu's presentation of $K_0(\sM_2)$ \cite[Theorem~1.1]{EH24}. We also work out the presentations of $K_0(\sM_2)$ and $K_0(\overline{\sM}_2)$ in terms of $K$-theoretic Chern classes in \Cref{app:edidinhu}.

\subsection{Odd genus}\label{sec:odd}

Throughout this subsection, let $g$ be odd, and recall the rank-three bundle $\cE:=\pi_*\omega_{\sC/\sH_g}^{-1}(\sW)$ and the line bundle $\cL:=\pi_*\omega_{\sC/\sH_g}^{\otimes(g+1)/2}(\tfrac{1-g}{2}\sW)$ on $\sH_g$, where $\pi\colon\sC\to\sH_g$ is the universal curve and $\sW$ is its divisor of Weierstrass points \cite[Section~6]{DL21}. Let $\sigma=[\cE]$ and $\lambda=[\cL]$. Under the isomorphism $\GL_2/\mmu{g+1}\cong\bG_m\times\PGL_2$, $[A]\mapsto(\det(A)^{(g+1)/2},[A])$, of \cite[Proposition~4.4(a)]{AV04}, these bundles correspond to the three-dimensional irreducible representation of $\PGL_2$ and the weight-one character of $\bG_m$, respectively. Their pullbacks to $\sX_g$ are therefore $\Sym^2\cU\otimes(\det\cU)^{-1}$ and $(\det\cU)^{(g+1)/2}$, so we use the same notation for the corresponding elements of $R(\GL_2)$:
\[
\sigma := \xi^2\psi^{-1}-1,\qquad \lambda := \psi^{(g+1)/2}.
\]

\begin{proposition}\label{prop:oddlattice}\raggedright
For odd $g$, the $\mmu{g+1}$-weight-zero subring of $\bZ[\xi, \psi^{\pm1}]$ is $\bZ[\sigma, \lambda^{\pm1}]$, and the weight-$(-1)$ module is $\xi\psi^{(g-1)/2}\,\bZ[\sigma, \lambda^{\pm1}]$.
\end{proposition}

\begin{proof}
A weight-zero monomial $\xi^i\psi^j$ satisfies $i+2j\equiv0\pmod{g+1}$.  Since $g+1$ is even, we have $i=2m$ and $(g+1)/2$ divides $m+j$.  Thus,
\[
\xi^{2m}\psi^j
=(\xi^2\psi^{-1})^m\psi^{m+j}
=(\sigma+1)^m\lambda^{2(m+j)/(g+1)}.
\]
Since $\sigma$ and $\lambda^{\pm1}$ have weight zero, the above equality identifies the weight-zero subring with $\bZ[\sigma,\lambda^{\pm1}]$. Every monomial of weight $-1$ has an odd exponent of $\xi$, so division by $\xi\psi^{(g-1)/2}$ gives a monomial of weight zero. Hence, the weight-$(-1)$ module is $\xi\psi^{(g-1)/2}\bZ[\sigma,\lambda^{\pm1}]$.
\end{proof}

To express $F_g$ and $\xi\psi^{(g-1)/2}G_g$ in $\sigma$ and $\lambda^{\pm1}$, we use the irreducible representations of $\PGL_2$.  The torus characters of $\Sym^2\cU\otimes(\det\cU)^{-1}$ are $\alpha/\beta, 1, \beta/\alpha$, so $\sigma = \alpha/\beta + 1 + \beta/\alpha$.  By \Cref{eq:Qclosed}, the character $\sum_{j=-k}^k(\alpha/\beta)^j$ of the $(2k+1)$-dimensional irreducible representation of $\PGL_2$ equals $Q_k(\sigma)$.

\begin{theorem}\label{thm:oddforms}\label{thm:odd}
Let $g\ge3$ be odd. We have the relations
\[
F_g \;=\; F^{\mathrm{odd}}_g \;:=\;
1 + \lambda^2(Q_g(\sigma) + Q_{g-1}(\sigma)) - \lambda^4Q_{2g}(\sigma),
\qquad
\xi\psi^{(g-1)/2}\,G_g \;=\; \lambda\,G^{\mathrm{odd}}_g,
\]
where $G^{\mathrm{odd}}_g := (\sigma+1)\bigl(1 + \lambda^2Q_{g-1}(\sigma)\bigr) - \lambda^4(Q_{2g}(\sigma) + Q_{2g-1}(\sigma))$. Consequently,
\[
K_0(\sH_g) \;\cong\;
\bZ[\sigma, \lambda^{\pm1}]\big/(F^{\mathrm{odd}}_g, G^{\mathrm{odd}}_g).
\]
This is a free abelian group of rank $8g+4$, with $\sigma=[\cE]$ and $\lambda=[\cL]$.
\end{theorem}

\begin{proof}
The polynomial identities \eqref{eq:Pclosed} and \eqref{eq:Qclosed} give $P_{2k}(\xi,\psi)=\psi^kQ_k(\sigma)$ for $k\ge0$, and the recurrence for $P$ then gives
\[
\xi P_{2k+1}(\xi,\psi)=P_{2k+2}(\xi,\psi)+\psi P_{2k}(\xi,\psi)
=\psi^{k+1}\bigl(Q_{k+1}(\sigma)+Q_k(\sigma)\bigr).
\]
Substituting these identities into $F_g$ and $G_g$ with $\xi^2=\psi(\sigma+1)$ and $\lambda=\psi^{(g+1)/2}$ yields $F_g=F_g^{\mathrm{odd}}$ and $\xi\psi^{(g-1)/2}G_g=\lambda G_g^{\mathrm{odd}}$. By \eqref{eq:parityplan} and \Cref{prop:oddlattice}, the kernel of $\bZ[\sigma,\lambda^{\pm1}]\to K_0(\sH_g)$ is $(F_g^{\mathrm{odd}},\lambda G_g^{\mathrm{odd}})=(F_g^{\mathrm{odd}},G_g^{\mathrm{odd}})$, since $\lambda$ is a unit, proving the desired presentation. Finite freeness and the rank then follow from \Cref{prop:characterranks}.
\end{proof}

As in even genus, we obtain the following description of $\Spec(K_0(\sH_g)\otimes\overline{\bQ})$ by restricting the points of $\Spec(A_g\otimes\overline{\bQ})$ to the weight-zero subring. The characters $\sigma=\xi^2\psi^{-1}-1$ and $\lambda=\psi^{(g+1)/2}$ send $\{\alpha,\beta\}$ to
\[
(\sigma,\lambda)=\bigl(\alpha/\beta+1+\beta/\alpha,\ (\alpha \beta)^{(g+1)/2}\bigr).
\]
Thus, $\sigma$ determines the ratio $\rho=\alpha/\beta$ up to taking its inverse, and the equations in \Cref{prop:masterspec} give the following possibilities for $\rho$ and $\lambda$.

\begin{proposition}\label{prop:oddspectrum}
Let $g$ be odd.  The scheme $\Spec(K_0(\sH_g)\otimes\overline{\bQ})$ consists of $8g+4$ reduced points $(\sigma,\lambda)=(\rho+1+\rho^{-1},\lambda)$, with $(\rho,\lambda)$ satisfying one of the following conditions:
\begin{enumerate}[leftmargin=2.4em]
\item $\lambda = \pm1$ and $\rho\in\mmu{g}\cup\mmu{g+1}$ with $\rho$ and $\rho^{-1}$ identified, giving $2(g+1)$ points.
\item $\lambda = \pm i$ and $\rho^{\,g} = -1$ or $\rho^{\,g+1} = -1$, with $\rho$ and $\rho^{-1}$ identified, giving $2(g+1)$ points.
\item $\lambda^2\in\{\rho^{\,g}, \rho^{\,g+1}\}$ with $\rho\in\mmu{2g+1}\smallsetminus\{1\}$, identifying $\rho$ with $\rho^{-1}$, giving $4g$ points with $\lambda^4\neq1$, on which $\sigma + 1 = (1+\lambda^4)^2\lambda^{-4}$.
\end{enumerate}
The elimination ideal $((F_g^{\mathrm{odd}},G_g^{\mathrm{odd}})\otimes\bQ)\cap\bQ[\lambda^{\pm1}]$ is generated by $(\lambda^2+1)(\lambda^{4g+2}-1)$. Moreover, $(\lambda^2+1)(\lambda^{4g+2}-1)\in (F_g^{\mathrm{odd}},G_g^{\mathrm{odd}})\subset\bZ[\sigma,\lambda^{\pm1}]$.
\end{proposition}

\begin{proof}
With $\rho=\alpha/\beta$, the three conditions on $\{\alpha,\beta\}$ take the form \Cref{eq:restriction-conditions}, and
\[
\lambda^2=\rho^{g+1}\beta^N.
\]
When $\rho^N=\beta^N=1$, we have $\lambda^2=\rho^{g+1}\in\{\pm1\}$, whereas when $\rho^{N-2}=1$, $\rho\ne1$, and $\beta^N=\rho^{-1}$, we have $\lambda^2=\rho^g\in\{\pm1\}$. Together, these give the first two cases of the statement. When $\rho^{N-1}=1$, $\rho\ne1$, and $\beta^N=\rho^{-1}$, we have $\lambda^2=\rho^g$, with $\rho\in\mmu{2g+1}\smallsetminus\{1\}$. Replacing $\rho$ by $\rho^{-1}$ gives the alternative $\lambda^2=\rho^{g+1}$. For fixed $\rho$, the choices of $\beta$ give both signs of $\lambda$, so these are the points of the spectrum.

The union $\mmu g\cup\mmu{g+1}$ has $2g$ elements, of which only $1$ and $-1$ equal their own inverses. It therefore gives $g+1$ values of $\sigma$ for each of $\lambda=\pm1$. The sets defined by $\rho^g=-1$ and $\rho^{g+1}=-1$ are disjoint and have total size $2g+1$, with only $-1$ equal to its own inverse, giving $g+1$ values of $\sigma$ for each of $\lambda=\pm i$. In the third case, each of the $g$ unordered pairs $\{\rho,\rho^{-1}\}$ gives four distinct choices of $\lambda$, all with $\lambda^4\ne1$. At these points, $\lambda^4\in\{\rho,\rho^{-1}\}$, so
\[
\sigma+1=\rho+2+\rho^{-1}=(1+\lambda^4)^2\lambda^{-4}.
\]

Since raising to the $g$-th power permutes $\mmu{2g+1}$, the values of $\lambda$ are $\mmu{4g+2}\cup\{\pm i\}$. Since $K_0(\sH_g)\otimes\bQ$ is reduced, a Laurent polynomial in $\bQ[\lambda^{\pm1}]$ vanishes in this ring if and only if it vanishes at all these values, or equivalently is divisible by $(\lambda^2+1)(\lambda^{4g+2}-1)$. Torsion-freeness of $K_0(\sH_g)$ then gives $(\lambda^2+1)(\lambda^{4g+2}-1)\in(F_g^{\mathrm{odd}},G_g^{\mathrm{odd}})$ over $\bZ$.
\end{proof}

\subsection{Connection with the Picard group}

We can now relate our $K_0$ computation to known facts about the Picard group of $\sH_g$.

\begin{proposition}\label{prop:odddisc}
The least positive integer $n$ such that $\lambda^n=1$ in $K_0(\sH_g)$ is $4g+2$ for even $g$, with $\lambda=[\det\cV]$, and $8g+4$ for odd $g$, with $\lambda=[\cL]$.
\end{proposition}

\begin{proof}
The identity $\psi^{N(N-1)}=1$ gives $\lambda^{4g+2}=1$ for even $g$, since $\lambda=\psi^{g+1}$, and $\lambda^{8g+4}=1$ for odd $g$, since $\lambda=\psi^{(g+1)/2}$.

For minimality, \Cref{prop:evenspectrum,prop:oddspectrum} show that the values of $\lambda$ include a primitive $(2g+1)$-th root of unity in both parities, as well as $-1$ for even $g$ and $i$ for odd $g$. Hence, any exponent $n$ with $\lambda^n=1$ is divisible by $\operatorname{lcm}(2,2g+1)=4g+2$ for even $g$ and by $\operatorname{lcm}(4,2g+1)=8g+4$ for odd $g$.
\end{proof}

\begin{remark}\label{rem:pic}
Taking determinants of the bundle classes in \Cref{thm:intromain} shows that $\Pic(\sH_g)$ is generated by $\det\cV$ for even $g$ and by $\cL$ for odd $g$, since $\det\cE\cong\cO$. Sending a line bundle to its $K_0$-class is injective, as taking the determinant recovers the line bundle. Thus, the order of $\lambda$ in \Cref{prop:odddisc} recovers the Picard group \cites[Section~5]{AV04}[Section~6]{DL21}:
\[
\Pic(\sH_g)\cong
\begin{cases}
\bZ/(4g+2),&g\text{ even},\\
\bZ/(8g+4),&g\text{ odd}.
\end{cases}
\]
\end{remark}

\begin{remark}\label{rem:basefield}
Although the binary-form presentation of $\sH_g$ holds in characteristic different from $2$ \cite[Example~3.5(a)]{AV04}, our presentation of $K_0(\sH_g)$ fails in characteristic $3$, already for $g=2$.

Over a field of characteristic $3$, the $\GL_2$-invariant linear subspace
\[
W=\Bbbk\cdot\{x_0^6,x_0^3x_1^3,x_1^6\}\subset\bA^7
\]
consists of cubes of quadratic forms, so $W\subset\Delta$. The dual of its normal representation has torus weights $\alpha \beta^5,\alpha^2\beta^4,\alpha^4\beta^2,\alpha^5\beta$, and the Koszul resolution gives the relation $\prod_{i\in\{1,2,4,5\}}(1-\alpha^i\beta^{6-i})=0$ on the discriminant complement. This class has $\mmu3$-weight zero and therefore gives a relation on $\sH_2$. In the generators $\nu=(\alpha+\beta)\alpha \beta$ and $\lambda=(\alpha \beta)^3$, it is
\begin{equation}\label{eq:charthreeextra}
\bigl(1-\nu^2+2\lambda+\lambda^2\bigr)
\bigl(1-\lambda^{-1}\nu^4+4\nu^2-2\lambda+\lambda^2\bigr)=0.
\end{equation}
Both defining relations in \Cref{thm:even} vanish at $(\nu,\lambda)=(-1,1)$, whereas the left-hand side of \Cref{eq:charthreeextra} evaluates to $9$ in $\bQ$. Thus, this additional relation is nonzero in the characteristic-zero presentation, which therefore fails in characteristic $3$.
\end{remark}

\section{Inertia, Riemann--Roch, and {$\operatorname{Spec}$} of $K_0$ and Chow Rings}
\label{sec:inertia}

Throughout this section, we work over an algebraically closed field of characteristic zero. By Riemann--Roch, the rank of $K_0(\sH_g)$ is the sum of the dimensions of the rational Chow rings of the inertia components (\Cref{thm:introlengths}). Edidin--Hu used this comparison to determine the rank of $K_0(\overline{\sM}_2)$ \cite[Section~6.1]{EH24}. For $\sH_g$, Pagani showed that $CH^\bullet(\sZ)\otimes\bQ\cong\bQ$ for every inertia component $\sZ\subset\sI\sH_g$ \cite[Theorem~4.2 and Remark~5.2]{Pag12}, so $\operatorname{rank}_{\bZ}K_0(\sH_g)=\#\pi_0(\sI\sH_g)$.

We also compare $K_0(\sH_g)$ with $CH^\bullet(\sH_g)$ through the gamma-filtration. For every prime $p$, we construct a surjection from $CH^\bullet(\sH_g)\otimes\bF_p$ to $\operatorname{gr}_F(K_0(\sH_g)\otimes\bF_p)$ and give a presentation of the associated graded ring, where $F$ is the image of the gamma-filtration.

\subsection{The rank of \texorpdfstring{$K_0$}{K0} and the inertia stack}\label{subsec:lengths}

\begin{proof}[Proof of \Cref{thm:introlengths}]
By \Cref{thm:inertiachern}, $K_0(\mathscr X)\otimes\overline{\bQ}$ is the product of the rings $CH^\bullet(\mathscr Z)\otimes\overline{\bQ}$ over the connected components $\mathscr Z$ of $\sI\mathscr X$. Each ring $CH^\bullet(\mathscr Z)\otimes\overline{\bQ}$ is finite-dimensional, with unique maximal ideal consisting of the positive-codimension classes and residue field $\overline{\bQ}$. Summing the lengths of these factors gives
\[
r=\sum_{\mathscr Z\in\pi_0(\sI\mathscr X)}\dim_{\bQ}(CH^\bullet(\mathscr Z)\otimes\bQ).
\]
Hence $r=\#\pi_0(\sI\mathscr X)$ precisely when $CH^\bullet(\mathscr Z)\otimes\overline{\bQ}\cong\overline{\bQ}$ for every $\mathscr Z$, or equivalently when $K_0(\mathscr X)\otimes\bQ$ is reduced. Applying this to the component $\mathscr X\subset\sI\mathscr X$ consisting of identity automorphisms gives $CH^\bullet(\mathscr X)\otimes\bQ\cong\bQ$. Finally, freeness of $K_0(\mathscr X)$ over $\bZ$ gives $\dim_{\bar{\Bbbk}}(K_0(\mathscr X)\otimes\bar{\Bbbk})=r$ for every algebraically closed coefficient field $\bar{\Bbbk}$, so every geometric fiber of $\Spec K_0(\mathscr X)\to\Spec\bZ$ has length $r$.
\end{proof}

\subsection{\texorpdfstring{On $\operatorname{Spec} K_0(\sH_g)$}{On Spec K0(Hg)}}
\label{subsec:spectra}

The inertia stack $\sI\sH_g$ parametrizes pairs $(C,\tau)$, where $C$ is a smooth hyperelliptic curve of genus $g$ and $\tau\in\operatorname{Aut}(C)$. Pagani's description of its connected components gives a geometric interpretation of the rank of $K_0(\sH_g)$.

\begin{proposition}\label{prop:sectorpoint}
Let the base field be algebraically closed of characteristic zero.  Every component $\sZ$ of $\sI\sH_g$ has $CH^\bullet(\sZ)\otimes\bQ\cong \bQ$, concentrated in degree zero.  Moreover,
\[
\dim_\bQ\bigl(K_0(\sH_g)\otimes\bQ\bigr)
=\#\pi_0(\sI\sH_g)
=\begin{cases}8g+2,&g\text{ even},\\8g+4,&g\text{ odd}.\end{cases}
\]
\end{proposition}

\begin{proof}
Pagani obtained $CH^\bullet(\sZ)\otimes\bQ\cong\bQ$ by describing the coarse moduli spaces of the components as finite quotients of $M_{0,n}$ \cite[Theorem~4.2 and Remark~5.2]{Pag12}. Since $M_{0,n}\subset\bA^{n-3}$ is open, the same conclusion holds over our base field \cite[Properties~7.1.36(2)--(5)]{Alp26}. Counting the components in his classification and applying \Cref{thm:introlengths} gives the two equalities.
\end{proof}

By \eqref{eq:inertiachern}, we have
\[
K_0(\sH_g)\otimes\overline{\bQ}\cong\prod_{\sZ\in\pi_0(\sI\sH_g)}\overline{\bQ},
\]
so each connected component of $\sI\sH_g$ corresponds to a point of $\Spec(K_0(\sH_g)\otimes\overline{\bQ})$. The identity and the hyperelliptic involution $\iota$ give two components isomorphic to $\sH_g$, and projection to the component of identity automorphisms is the rank homomorphism.

For a pair $(C,\tau)$ on any other component, the induced automorphism of $C/\langle\iota\rangle\cong\bP^1$ is nontrivial. Choose coordinates so that this automorphism is $x\mapsto\rho x$, where $\rho$ has order $d\ge2$. The branch points away from $(0:1)$ and $(1:0)$ occur in orbits of size $d$, so the binary form defining $C$ has the form \cite[proof of Theorem~4.2]{Pag12}
\begin{equation}\label{eq:fixedbinaryform}
f=x_0^{n_1}x_1^{n_2}u(x_0^d,x_1^d),
\qquad
N=md+n_1+n_2.
\end{equation}
Here $n_1,n_2\in\{0,1\}$ are the orders of vanishing of $f$ at $(0:1)$ and $(1:0)$, respectively, while $u$ is a square-free binary form of degree $m\ge1$ with $u(1,0)u(0,1)\ne0$.

After interchanging $x_0,x_1$ if necessary, the possibilities are $(n_1,n_2)=(0,0),(1,0),(1,1)$. On $y^2=f(x,1)$, the action $x\mapsto\rho x$ determines $\tau$ up to composition with $\iota$, which changes the sign of $y$. With $\zeta$ a primitive $2d$-th root of unity satisfying $\zeta^2=\rho$, Pagani's formulas give
\begin{equation}\label{eq:binarystabilizers}
\begin{array}{c|c|c|c}
f & \text{condition} & \tau(x,y) & \langle\iota,\tau\rangle\\[2pt]
\hline
u(x_0^d,x_1^d) & d\mid 2g+2 & (\rho x,\pm y)
& \mmu d\times\mmu2\\[2pt]
x_0u(x_0^d,x_1^d) & d\mid 2g+1 & (\rho x,\pm\zeta y)
& \mmu{2d}\\[2pt]
x_0x_1u(x_0^d,x_1^d) & d\mid 2g & (\rho x,\pm\zeta y)
& \mmu{2d}
\end{array}
\end{equation}
The subgroup $\langle\iota,\tau\rangle\subset\operatorname{Stab}_{\GL_2/\mmu{g+1}}(f)$ consists of the automorphisms inducing powers of $x\mapsto\rho x$ on $\bP^1$.

The difference between the two ranks in \Cref{prop:sectorpoint} comes from the components with $d=2$. Pagani's classification gives two such components when $g$ is even and four when $g$ is odd, with $8g$ other components in either parity \cite[Theorem~4.2]{Pag12}.

\subsection{The Chow ring and the gamma-filtration}
\label{subsec:chernpresentations}\label{subsec:lowest}\label{subsec:reduction}

Although every fiber of $\Spec K_0(\sH_g)\to\Spec\bZ$ has finite length, the corresponding fibers for $CH^\bullet(\sH_g)$ can have positive dimension. We first describe these fibers, then compute the quotients of the Chow rings given by the gamma-filtration. For the Chow rings, we use the presentations \eqref{eq:chow} with $c_i=c_i(\cV)$ for even $g$, and $t=c_1(\cL)$ and $c_i=c_i(\cE)$ for odd $g$.

\begin{proposition}\label{prop:chowfibers}
For $\Bbbk=\bQ$ or $\Bbbk=\bF_p$ with $p$ prime, the fiber of $\Spec CH^\bullet(\sH_g)\to\Spec\bZ$ is
\[
\begin{array}{c|c|c}
\text{coefficient field }\Bbbk & g\text{ even} & g\text{ odd}\\ \hline
\bF_2 & \bA^2 & \bA^3\\
\bF_p,\ p\text{ odd},\ p\mid 2g+1
& \{g(g-1)c_1^2=4g(g+1)c_2\}
& \{8t^2=2g(g+1)c_2,\ c_3=0\}\\
\bF_p,\ p\text{ odd},\ p\mid g(g+1)
& \{c_1=0\}\cong\bA^1
& \{t=c_3=0\}\cong\bA^1\\
\bQ\text{ or any other }\bF_p
& \Spec \Bbbk & \Spec \Bbbk
\end{array}
\]
where each scheme is over the field in its row. In the second row, the fiber for even $g$ is a parabola for $p\neq3$ and the line $c_2=0$ for $p=3$.  The fiber for odd $g$ is a parabola. In particular, a fiber has positive dimension exactly at the primes dividing $2g(g+1)(2g+1)$.
\end{proposition}

\begin{proof}
The table follows from \eqref{eq:chow} by extension of scalars to $\Bbbk$. For odd $p\mid2g+1$, the congruences $4g(g+1)\equiv-1$ and $g(g-1)\equiv3/4\pmod p$ show that the coefficient of $c_2$ is nonzero, whereas that of $c_1^2$ vanishes only for $p=3$.
\end{proof}

To determine the gamma-filtration, we use $c_1^K(\cV)=\delta_1+\delta_2$ and $c_2^K(\cV)=\delta_2$ from \eqref{eq:kchern-ranktwo} in even genus, and the following formulas in odd genus.

\begin{lemma}\label{lem:cEchern}Let $g$ be odd.  The bundle $\cE$ is self-dual with trivial determinant, and its $K$-theoretic Chern classes are
\[
c_1^K(\cE) \;=\; c_2^K(\cE) \;=\; 3 - \sigma,
\qquad
c_3^K(\cE) \;=\; 0 .
\]
\end{lemma}

\begin{proof}
The bundle $\cE$ descends from $\Sym^2\cU\otimes\det\cU^\vee$, which is self-dual with trivial determinant. Thus $\wedge^2\cE^\vee\cong\cE$, and the exterior-power formula gives
\[
c_t^K(\cE)
=(1+t)^3-\sigma t(1+t)^2+\sigma t^2(1+t)-t^3
=1+(3-\sigma)t+(3-\sigma)t^2.
\qedhere
\]
\end{proof}

We obtain the gamma-filtration on $K_0(\sH_g)$ from the filtrations on $R(\GL_2)$ for even $g$ and $R(\bG_m\times\PGL_2)$ for odd $g$.

\begin{proposition}\label{prop:gamma}
For every $g\ge2$ and every $n\ge0$, we have
\[
F_\gamma^nK_0(\sH_g)=
\begin{cases}
(\delta_1^i\delta_2^j: i,j\ge0,\ i+2j\ge n),
&g\text{ even},\\
(\tau_1^i\kappa_2^j: i,j\ge0,\ i+2j\ge n),
&g\text{ odd}.
\end{cases}
\]
\end{proposition}

\begin{proof}
For even $g$, the ideal $F_\gamma^nR(\GL_2)$ is generated by monomials in $c_1^K(\cV),c_2^K(\cV)$ of weighted degree at least $n$, with weights $1,2$ \cite[Proposition~4.14, Example~5.3, and Appendix~G]{KM21}. Since $c_1^K(\cV)=\delta_1+\delta_2$ and $c_2^K(\cV)=\delta_2$, this ideal is also $(\delta_1^i\delta_2^j:i,j\ge0,\ i+2j\ge n)$.

For odd $g$, we use the isomorphism of augmented $\lambda$-rings $R(\PGL_2)\cong R(\mathrm{SL}_2)$ that sends $\sigma-1$ to the class of the standard $\mathrm{SL}_2$-representation. Under this isomorphism, $\kappa_2=3-\sigma$ corresponds to the second $K$-theoretic Chern class of that representation, so $F_\gamma^nR(\PGL_2)=(\kappa_2^{\lceil n/2\rceil})$ by \cite[Section~5 and Appendix~S]{KM21}. Together with $F_\gamma^nR(\bG_m)=(\tau_1^n)$ for $\tau_1=1-\lambda^{-1}$, this gives
\[
F_\gamma^nR(\bG_m\times\PGL_2)
=(\tau_1^i\kappa_2^j:i,j\ge0,\ i+2j\ge n).
\]

The maps $R(\GL_2)\twoheadrightarrow K_0(\sH_g)$ for even $g$ and $R(\bG_m\times\PGL_2)\twoheadrightarrow K_0(\sH_g)$ for odd $g$ in \Cref{thm:even,thm:odd} are surjections of augmented $\lambda$-rings, so they induce surjections on every term of the gamma-filtration.
\end{proof}

\begin{proof}[Proof of \Cref{thm:introleading}]
By \Cref{prop:gamma}, the graded classes of $\delta_1,\delta_2$ for even $g$, and $\tau_1,\kappa_2$ for odd $g$, generate $\operatorname{gr}_\gamma K_0(\sH_g)$. The lowest-degree terms in \Cref{thm:introfc} give the relations among $c_1,c_2$ or $t,c_2$ in \eqref{eq:chow}, so sending the Chow generators to these graded classes defines the surjections. For odd $g$, the map factors through $CH^\bullet(\sH_g)/(c_3)$ because $c_3^K(\cE)=0$ by \Cref{lem:cEchern}.
\end{proof}

After extension of scalars to $\bQ$ or $\bF_p$, we can compute the associated graded ring by localizing at the kernel of the rank homomorphism.

\begin{proposition}\label{prop:augmentation}
Let $\Bbbk=\bQ$ or $\Bbbk=\bF_p$ with $p$ prime, and let $\mathfrak m=\ker(\operatorname{rank}\colon K_0(\sH_g)\otimes \Bbbk\to \Bbbk)$. Give $K_0(\sH_g)\otimes \Bbbk$ the filtration
\[
F^n\bigl(K_0(\sH_g)\otimes \Bbbk\bigr)
=\operatorname{im}\bigl(F_\gamma^nK_0(\sH_g)\otimes \Bbbk\longrightarrow K_0(\sH_g)\otimes \Bbbk\bigr).
\]
There are natural surjections
\begin{equation}\label{eq:chowreduction}
\begin{aligned}
CH^\bullet(\sH_g)\otimes \Bbbk
&\twoheadrightarrow\operatorname{gr}_F\bigl(K_0(\sH_g)\otimes \Bbbk\bigr)
&& (g\text{ even}),\\
\bigl(CH^\bullet(\sH_g)/(c_3)\bigr)\otimes \Bbbk
&\twoheadrightarrow\operatorname{gr}_F\bigl(K_0(\sH_g)\otimes \Bbbk\bigr)
&& (g\text{ odd}),
\end{aligned}
\end{equation}
sending $c_1,c_2$ to the classes of $\delta_1,\delta_2$ in even genus and $t,c_2$ to those of $\tau_1,\kappa_2$ in odd genus. Moreover, localization gives
\[
\begin{aligned}
\operatorname{gr}_F\bigl(K_0(\sH_g)\otimes \Bbbk\bigr)
&\cong\operatorname{gr}_F\bigl((K_0(\sH_g)\otimes \Bbbk)_{\mathfrak m}\bigr),\\
\dim_\Bbbk\operatorname{gr}_F\bigl(K_0(\sH_g)\otimes \Bbbk\bigr)
&=\dim_\Bbbk\bigl(K_0(\sH_g)\otimes \Bbbk\bigr)_{\mathfrak m}.
\end{aligned}
\]
\end{proposition}

\begin{proof}
Tensoring the surjection of \Cref{thm:introleading} with $\Bbbk$ and composing with the natural surjection
\begin{equation}\label{eq:gradedbasechange}
\bigl(\operatorname{gr}_\gamma K_0(\sH_g)\bigr)\otimes \Bbbk
\twoheadrightarrow\operatorname{gr}_F\bigl(K_0(\sH_g)\otimes \Bbbk\bigr)
\end{equation}
gives \eqref{eq:chowreduction}.

By \Cref{prop:gamma}, the ideal $\mathfrak m$ is $(\delta_1,\delta_2)$ for even $g$ and $(\tau_1,\kappa_2)$ for odd $g$, and
\[
\mathfrak m^n\subseteq F^n\bigl(K_0(\sH_g)\otimes \Bbbk\bigr)
\subseteq\mathfrak m^{\lceil n/2\rceil}.
\]
At every maximal ideal other than $\mathfrak m$, the filtration is constant and has zero associated graded ring. The Artinian product decomposition \cite[Tag~00JB]{Stacks} therefore shows that localization at $\mathfrak m$ preserves the associated graded ring. There the filtration terminates in zero, so its associated graded ring has dimension $\dim_\Bbbk(K_0(\sH_g)\otimes\Bbbk)_{\mathfrak m}$.
\end{proof}

The surjection in \Cref{eq:gradedbasechange} need not be injective. For example, $R(\mmu2)\cong\bZ[z]/(z^2-2z)$, where $z=1-[\mathrm{sgn}]$ for the nontrivial character $\mathrm{sgn}$, and the gamma-filtration is given by powers of $(z)$. Thus,
\[
\bigl(\operatorname{gr}_\gamma R(\mmu2)\bigr)\otimes\bF_2\cong\bF_2[t],
\qquad
\operatorname{gr}_F\bigl(R(\mmu2)\otimes\bF_2\bigr)\cong\bF_2[t]/(t^2).
\]

After reducing the defining polynomials in \Cref{thm:introfc} modulo $p$, we take the lowest-degree terms of their combinations to obtain relations in $\operatorname{gr}_F(K_0(\sH_g)\otimes\bF_p)$. We prove that these relations suffice by comparing dimensions using \Cref{prop:augmentation}. We compute this dimension by counting the points of $\Spec(K_0(\sH_g)\otimes\overline{\bQ})$ specializing to $\mathfrak m$. For a positive integer $n$, let $v_p(n)$ denote the exponent of $p$ in $n$.

\begin{proposition}\label{prop:local-lengths}
Let $g\ge2$, let $p$ be a prime, and let $\mathfrak m=\ker(\operatorname{rank}\colon K_0(\sH_g)\otimes\bF_p\to\bF_p)$. Then
\[
\dim_{\bF_p}\bigl(K_0(\sH_g)\otimes\bF_p\bigr)_{\mathfrak m}=
\begin{cases}
p^{v_p(2g+1)}+\dfrac{p^{v_p(g(g+1))}-1}{2},&p\ne2,\\[5pt]
2^{v_2(g)+1}+2,&p=2,\ g\text{ even},\\[2pt]
2^{v_2(g+1)+1}+4,&p=2,\ g\text{ odd}.
\end{cases}
\]
For $\mathfrak m=\ker(\operatorname{rank}\colon K_0(\sH_g)\otimes\bQ\to\bQ)$, we have $(K_0(\sH_g)\otimes\bQ)_{\mathfrak m}\cong\bQ$.
\end{proposition}

\begin{proof}
Choose an extension of the $p$-adic valuation to $\overline{\bQ}$. Since $\bZ_p$ is henselian, $(K_0(\sH_g)\otimes\bF_p)_{\mathfrak m}$ is the reduction modulo $p$ of a finite free $\bZ_p$-algebra that is a direct summand of $K_0(\sH_g)\otimes\bZ_p$ \cite[Tags~04GG and~02KB]{Stacks}. The rank of this summand equals both $\dim_{\bF_p}(K_0(\sH_g)\otimes\bF_p)_{\mathfrak m}$ and the number of points specializing to $\mathfrak m$, since $K_0(\sH_g)\otimes\overline{\bQ}$ is reduced.

A root of unity specializes to $1$ precisely when its order is a power of $p$. Thus, in even genus, specialization to $(\nu,\lambda)=(2,1)$ is equivalent to both eigenvalues having $p$-power order. In odd genus, it is equivalent to $\rho$ and $\lambda$ having $p$-power order, since $\sigma-3=(\rho-1)^2/\rho$.

For odd $p$, the points with $\lambda=1$ in \Cref{prop:evenspectrum,prop:oddspectrum} contribute $(p^{v_p(g(g+1))}+1)/2$, by pairing each root of $p$-power order with its inverse. Among the points with $\lambda^4\ne1$, there are $(p^{v_p(2g+1)}-1)/2$ unordered pairs $\{\rho,\rho^{-1}\}$ of nonidentity roots of $p$-power order. In the coordinates of \Cref{prop:evenspectrum}, $\beta^2$ takes the values $\zeta^{g-1},\zeta^g$, while in those of \Cref{prop:oddspectrum}, $\lambda^2$ takes the values $\rho^g,\rho^{g+1}$. For each value, exactly one of its two square roots specializes to $1$, so each pair $\{\rho,\rho^{-1}\}$ produces two points specializing to $\mathfrak m$. No point with $\lambda=-1$ or $\lambda=\pm i$ specializes to $\mathfrak m$, so the total number is
\[
\frac{p^{v_p(g(g+1))}+1}{2}+p^{v_p(2g+1)}-1
=p^{v_p(2g+1)}+\frac{p^{v_p(g(g+1))}-1}{2}.
\]

For $p=2$ and even $g$, the roots of $2$-power order in $\mmu{2g}\cup\mmu{2g+2}$ are $\mmu{2^{v_2(g)+1}}$. Pairing each root with its inverse gives $2^{v_2(g)}+1$ points with $\lambda=1$. For $\lambda=-1$, the solutions specializing to $1$ are $\zeta=\pm1$ and the roots of order $2^{v_2(g)+2}$. The substitution $\zeta\mapsto-\zeta^{-1}$ acts freely on them, again giving $2^{v_2(g)}+1$ points. No point with $\lambda\ne\pm1$ contributes, since $2g+1$ is odd.

For $p=2$ and odd $g$, the roots of $2$-power order in $\mmu g\cup\mmu{g+1}$ are $\mmu{2^{v_2(g+1)}}$. Only $1$ and $-1$ equal their own inverses, so the two choices $\lambda=\pm1$ give $2^{v_2(g+1)}+2$ points in total. For $\lambda=\pm i$, the allowed ratios are $\rho=-1$ and the roots of order $2^{v_2(g+1)+1}$. Only $-1$ equals its own inverse, and the two choices of $\lambda$ again give $2^{v_2(g+1)}+2$ points. Again, no point with $\lambda^4\ne1$ contributes, since $2g+1$ is odd. Finally, $(K_0(\sH_g)\otimes\bQ)_{\mathfrak m}\cong\bQ$, since $K_0(\sH_g)\otimes\bQ$ is reduced and the residue field at $\mathfrak m$ is $\bQ$.
\end{proof}

\begin{theorem}\label{thm:graded-prime-coefficients}
Let $g\ge2$, let $p$ be a prime, and let $F$ be the image of the gamma-filtration on $K_0(\sH_g)\otimes\bF_p$. For odd $p$, the associated graded ring $\operatorname{gr}_F(K_0(\sH_g)\otimes\bF_p)$ has the following presentations:
\[
\begin{array}{c|c|c}
&g\text{ even}&g\text{ odd}\\ \hline
p\mid2g+1
&\bF_p[\delta_1]/(\delta_1^{p^{v_p(2g+1)}})
&\bF_p[\tau_1]/(\tau_1^{p^{v_p(2g+1)}})\\[3pt]
p\mid g(g+1)
&\bF_p[\delta_2]/(\delta_2^{(p^{v_p(g(g+1))}+1)/2})
&\bF_p[\kappa_2]/(\kappa_2^{(p^{v_p(g(g+1))}+1)/2})\\[3pt]
p\nmid g(g+1)(2g+1)&\bF_p&\bF_p.
\end{array}
\]
For $p=2$, we have
\[
\operatorname{gr}_F\bigl(K_0(\sH_g)\otimes\bF_2\bigr)\cong
\begin{cases}
\bF_2[\delta_1,\delta_2]/(\delta_1^2,\delta_1\delta_2,
\delta_2^{\,2^{v_2(g)+1}+1}),&g\text{ even},\\[3pt]
\bF_2[\tau_1,\kappa_2]/(\tau_1^4,\tau_1^2\kappa_2+\kappa_2^2),
&g\equiv1\pmod4,\\[3pt]
\bF_2[\tau_1,\kappa_2]/(\tau_1^4,\tau_1^2\kappa_2,
\kappa_2^{\,2^{v_2(g+1)}+1}),&g\equiv3\pmod4.
\end{cases}
\]
The generators on the right map to the associated graded classes of the corresponding $K$-theoretic Chern classes.
\end{theorem}

\begin{proof}
Let $\mathfrak m=\ker(\operatorname{rank}\colon K_0(\sH_g)\otimes\bF_p\to\bF_p)$. By \Cref{prop:augmentation}, we may work in $\bigl(K_0(\sH_g)\otimes\bF_p\bigr)_{\mathfrak m}$, whose length is the dimension of the associated graded ring. For odd $p$, the surjections in \eqref{eq:chowreduction} and \Cref{prop:chowfibers} show that the associated graded ring is generated by $\delta_1$ in even genus or $\tau_1$ in odd genus when $p\mid2g+1$, and by $\delta_2$ or $\kappa_2$, respectively, when $p\mid g(g+1)$. Otherwise it is $\bF_p$. A finite-dimensional graded quotient of a polynomial ring in one variable is defined by a power of that variable, so \Cref{prop:local-lengths} determines the exponents.

For $p=2$, we use $\Phi_g,\Psi_g$ to denote the first two relations in the presentation in \Cref{thm:introfc}, reduced modulo $2$. Suppose first that $g$ is even. In characteristic two, the third defining polynomial in \eqref{eq:K0-Hg-chern-class} factors as
\[
\sum_{i=1}^{2g+1}\binom{2g+1}{i}\delta_1^{2i}
=\delta_1^2\sum_{j=0}^{2g}(1-\delta_1^2)^j.
\]
Since $\sum_{j=0}^{2g}(1-\delta_1^2)^j$ has constant term one, it is a unit in $(K_0(\sH_g)\otimes\bF_2)_{\mathfrak m}$, so $\delta_1^2=0$. The coefficient calculation in \Cref{app:chernproofs} gives
\begin{equation}\label{eq:even-mod-two-leading}
\Phi_g\equiv0,
\qquad \Phi_g+\Psi_g\equiv\delta_1\delta_2
\pmod{(\delta_1^2,\delta_2^2)}.
\end{equation}
Since $\delta_2^2$ has weighted degree four, the associated graded ring $\operatorname{gr}_F(K_0(\sH_g)\otimes\bF_2)$ is a quotient of $\bF_2[\delta_1,\delta_2]/(\delta_1^2,\delta_1\delta_2)$. The polynomials $\Phi_g,\Psi_g,\sum_{i=1}^{2g+1}\binom{2g+1}{i}\delta_1^{2i}$ have no terms of total degree one in $\delta_1,\delta_2$, so the class of $\delta_1$ in $F^1/F^2$ is nonzero. In higher degrees, only powers of $\delta_2$ can remain, and the length $2^{v_2(g)+1}+2$ forces the first vanishing power to be $\delta_2^{2^{v_2(g)+1}+1}$.

For odd $g$, the polynomials $\Phi_g,\Psi_g,\sum_{i=1}^{2g+1}\binom{2g+1}{i}\tau_1^{4i}$ belong to $\bF_2[\tau_1^2,\kappa_2]$. The relation $\sum_{i=1}^{2g+1}\binom{2g+1}{i}\tau_1^{4i}=0$ gives $\tau_1^4=0$ in $\bigl(K_0(\sH_g)\otimes\bF_2\bigr)_{\mathfrak m}$. The calculation in \Cref{app:chernproofs} gives
\begin{equation}\label{eq:odd-mod-two-leading}
\begin{aligned}
\Phi_g&\equiv0\pmod{(\tau_1^2,\kappa_2)^2},\\
\Psi_g&\equiv\tau_1^2\kappa_2+\frac{g+1}{2}\kappa_2^2
\pmod{(\tau_1^2,\kappa_2)^3}.
\end{aligned}
\end{equation}
If $g\equiv1\pmod4$, these relations give a surjection
\[
\bF_2[\tau_1,\kappa_2]/(\tau_1^4,\tau_1^2\kappa_2+\kappa_2^2)
\twoheadrightarrow\operatorname{gr}_F\bigl(K_0(\sH_g)\otimes\bF_2\bigr).
\]
The source is free with basis $1,\kappa_2$ over $\bF_2[\tau_1]/(\tau_1^4)$, so both sides have dimension eight by \Cref{prop:local-lengths}, and the surjection is an isomorphism.

For $g\equiv3\pmod4$, the ideal $(\Phi_g,\Psi_g,\sum_{i=1}^{2g+1}\binom{2g+1}{i}\tau_1^{4i})$ is extended from $\bF_2[\tau_1^2,\kappa_2]$, so $\bigl(K_0(\sH_g)\otimes\bF_2\bigr)_{\mathfrak m}$ is free with basis $1,\tau_1$ over its subalgebra generated by $\tau_1^2,\kappa_2$. Since $\tau_1^2$ and $\kappa_2$ both have weighted degree two, the decomposition respects the filtration. Thus, $\operatorname{gr}_F(K_0(\sH_g)\otimes\bF_2)$ is free with basis $1,\tau_1$ over its subalgebra generated by $\tau_1^2,\kappa_2$.

By \Cref{eq:odd-mod-two-leading}, the graded classes of $\tau_1^2,\kappa_2$ are linearly independent, with $(\tau_1^2)^2=\tau_1^2\kappa_2=0$, so the homogeneous component of degree $2j$ in this subalgebra is spanned by $\kappa_2^j$ for $j\ge2$. Its dimension is $\frac12\dim_{\bF_2}(K_0(\sH_g)\otimes\bF_2)_{\mathfrak m}=2^{v_2(g+1)}+2$, which forces the first vanishing power to be $\kappa_2^{2^{v_2(g+1)}+1}$. The free basis $1,\tau_1$ then gives the asserted presentation.
\end{proof}

\begin{remark}\label{prop:g3mod2}
For $g=3$, \Cref{thm:graded-prime-coefficients} gives
\[
\operatorname{gr}_F\bigl(K_0(\sH_3)\otimes\bF_2\bigr)
\cong\bF_2[\tau_1,\kappa_2]/(\tau_1^4,\tau_1^2\kappa_2,\kappa_2^5).
\]
This ring has dimension $12$, so twelve of the twenty-eight points of $\Spec(K_0(\sH_3)\otimes\overline{\bQ})$ specialize to the kernel of the rank homomorphism modulo $2$. The map from $CH^\bullet(\sH_3)\otimes\bF_2\cong\bF_2[t,c_2,c_3]$ has kernel $(c_3,t^4,t^2c_2,c_2^5)$. Thus, the associated graded ring defines a closed subscheme of length $12$ in the plane $c_3=0$ of $\Spec(CH^\bullet(\sH_3)\otimes\bF_2)\cong\bA^3$. \Cref{fig:g3-specialization} in \Cref{app:expanded} illustrates this specialization and a monomial basis for $\operatorname{gr}_F\bigl(K_0(\sH_3)\otimes\bF_2\bigr)$.
\end{remark}

\section{Future directions}\label{sec:future-directions}

Over a field of characteristic zero, the technique of isotypic decomposition that we used to compute $K_0(\sH_g)$ suggests further computations for the following moduli stacks.

\begin{enumerate}[leftmargin=2.4em]
\item \emph{Uniform cyclic covers of the projective line.} For $r\ge2$ and $d\ge1$ with $rd=2g+2$, Arsie--Vistoli's stack of smooth uniform cyclic covers has the presentation
\[
\sH_{\mathrm{sm}}(1,r,d)\cong
[(\bA^{2g+3}\setminus\Delta)/(\GL_2/\mmu d)]
\]
by \cite[Corollary~4.2]{AV04}. Its $K_0$-ring is therefore the trivial $\mmu d$-character subring of $K_0(\sX_g)$. Our presentation in \Cref{thm:eqpres} reduces its computation to finding generators and relations for this subring, which can then be compared with the Chow ring presentations of Fulghesu--Viviani \cite{FV11}.

\item \emph{K3 surfaces which are double covers of $\mathbb P^2$.} Let $U\subset H^0(\bP^2,\cO(6))$ be the open subset of forms defining smooth sextics. The stack of double covers of $\bP^2$ branched along a smooth sextic is $[U/(\GL_3/\mmu3)]$, and its objects are K3 surfaces \cite[Example~3.5(c) and Corollary~4.2]{AV04}. By \Cref{thm:isotypic}, its $K_0$-ring is the trivial $\mmu3$-character subring of $K_0(U,\GL_3)$. Therefore, it suffices to determine the relations in $K_0(U,\GL_3)$ arising from the locus of singular sextics.

\item \emph{Smooth hypersurfaces.} Let $U\subset\bP(H^0(\bP^n,\cO(m)))$ parametrize smooth hypersurfaces of degree $m\ge3$. The morphism $[U/\SL_{n+1}]\to[U/\PGL_{n+1}]$ is a $\mmu{n+1}$-gerbe, so computing $K_0(U,\SL_{n+1})$ and taking its trivial $\mmu{n+1}$-character subring would determine $K_0([U/\PGL_{n+1}])$. For plane quartics, the stratification of the discriminant used by Di~Lorenzo--Fulghesu--Vistoli in their Chow ring calculation \cite{DLFV21} provides a starting point for computing the relevant pushforwards in equivariant $K$-theory.

\item \emph{Stable vector bundles with fixed determinant.} Let $C$ be a smooth projective complex curve of genus at least two, and fix a line bundle $\cL$ on $C$ and an integer $r\ge2$. The stack $\sM$ of stable rank-$r$ bundles equipped with an isomorphism $\det\cE\cong\cL$ is a $\mmu r$-gerbe over its coarse moduli space \cite[Section~9.2.7]{Alp26}. The $\mmu r$-character summands of $K_0(\sM)$ are the $K_0$-groups of twisted sheaves \cite[Corollary~4.11]{BS21}. In rank two, it would be useful to combine this decomposition with Thaddeus's wall-crossing construction for stable pairs \cite{Tha94} to compute these groups and their products.
\end{enumerate}

\appendix

\section{Explicit calculations and examples}
\label{app:expanded}

\subsection{Presentations in genera two and three}
\label{app:edidinhu}
\label{app:expandedbundle}

Since $\sH_2\cong\sM_2$ and $\cV$ is the Hodge bundle in genus two, the classes $\nu,\lambda,P_n(\nu,\lambda)$ specialize to Edidin--Hu's classes $\epsilon,\lambda,h_n$. Thus, \Cref{thm:even} recovers their presentation \cite[Theorem~1.1]{EH24}.

In the bundle class generators $\nu,\lambda^{\pm1}$, the first relation is
\[
-\nu^{8}\,\lambda^{-2}+7\,\nu^{6}\,\lambda^{-1}-15\,\nu^{4}+\nu^{4}\,\lambda^{-1}+10\,\nu^{2}\,\lambda-2\,\nu^{2}-\lambda^{2}+1
\]
and the second is
\[
-\nu^{7}\,\lambda^{-1}+6\,\nu^{5}-10\,\nu^{3}\,\lambda+\nu^{3}+4\,\nu\,\lambda^{2}-\nu\,\lambda+\nu.
\]

Working over $\bC$, let $\cE$ be the Hodge bundle on $\overline{\sM}_2$, whose restriction to $\sM_2$ is $\cV$. Let
\[
\delta_1=c_1^K(\det\cE),\qquad
\delta_2=c_2^K(\cE).
\]
Taking duals sends $\nu$ to $\nu\lambda^{-1}$ and $\lambda$ to $\lambda^{-1}$. Substituting $\nu\lambda^{-1}=2-\delta_1-\delta_2$ and $\lambda^{-1}=1-\delta_1$ gives
\begingroup
\setlength{\fboxsep}{1pt}
\medmuskip=3mu plus 2mu minus 3mu
\allowdisplaybreaks[4]
\begin{align*}
&K_0(\sM_2)\cong\bZ[\delta_1,\delta_2]\big/\bigl(\\*
&\quad \colorbox{HighlightGold}{$\displaystyle\textcolor{RoyalPurple}{10\delta_1}$}-53\delta_1^{2}+96\delta_2+106\delta_1^{3}-360\delta_1\delta_2-119\delta_1^{4}+700\delta_1^{2}\delta_2-440\delta_2^{2}+83\delta_1^{5}-816\delta_1^{3}\delta_2+1342\delta_1\delta_2^{2}-36\delta_1^{6}\\*
&\quad +596\delta_1^{4}\delta_2-1952\delta_1^{2}\delta_2^{2}+784\delta_2^{3}+9\delta_1^{7}-270\delta_1^{5}\delta_2+1644\delta_1^{3}\delta_2^{2}-1968\delta_1\delta_2^{3}-\delta_1^{8}+70\delta_1^{6}\delta_2-825\delta_1^{4}\delta_2^{2}+2196\delta_1^{2}\delta_2^{3}\\*
&\quad -714\delta_2^{4}-8\delta_1^{7}\delta_2+231\delta_1^{5}\delta_2^{2}-1320\delta_1^{3}\delta_2^{3}+1429\delta_1\delta_2^{4}-28\delta_1^{6}\delta_2^{2}+420\delta_1^{4}\delta_2^{3}-1170\delta_1^{2}\delta_2^{4}+364\delta_2^{5}-56\delta_1^{5}\delta_2^{3}+455\delta_1^{3}\delta_2^{4}\\*
&\quad -546\delta_1\delta_2^{5}-70\delta_1^{4}\delta_2^{4}+294\delta_1^{2}\delta_2^{5}-105\delta_2^{6}-56\delta_1^{3}\delta_2^{5}+105\delta_1\delta_2^{6}-28\delta_1^{2}\delta_2^{6}+16\delta_2^{7}-8\delta_1\delta_2^{7}-\delta_2^{8},\\[3pt]
&\quad \colorbox{HighlightGold}{$\displaystyle\textcolor{RoyalPurple}{2\delta_1^{2}-24\delta_2}$}-\delta_1^{3}+59\delta_1\delta_2-107\delta_1^{2}\delta_2+194\delta_2^{2}+119\delta_1^{3}\delta_2-475\delta_1\delta_2^{2}-83\delta_1^{4}\delta_2+614\delta_1^{2}\delta_2^{2}-455\delta_2^{3}+36\delta_1^{5}\delta_2\\*
&\quad -477\delta_1^{3}\delta_2^{2}+980\delta_1\delta_2^{3}-9\delta_1^{6}\delta_2+225\delta_1^{4}\delta_2^{2}-987\delta_1^{2}\delta_2^{3}+494\delta_2^{4}+\delta_1^{7}\delta_2-60\delta_1^{5}\delta_2^{2}+550\delta_1^{3}\delta_2^{3}-879\delta_1\delta_2^{4}+7\delta_1^{6}\delta_2^{2}\\*
&\quad -165\delta_1^{4}\delta_2^{3}+660\delta_1^{2}\delta_2^{4}-286\delta_2^{5}+21\delta_1^{5}\delta_2^{3}-240\delta_1^{3}\delta_2^{4}+390\delta_1\delta_2^{5}+35\delta_1^{4}\delta_2^{4}-195\delta_1^{2}\delta_2^{5}+91\delta_2^{6}+35\delta_1^{3}\delta_2^{5}-84\delta_1\delta_2^{6}\\*
&\quad +21\delta_1^{2}\delta_2^{6}-15\delta_2^{7}+7\delta_1\delta_2^{7}+\delta_2^{8},\\[3pt]
&\quad \colorbox{HighlightGold}{$\displaystyle\textcolor{RoyalPurple}{10\delta_1}$}-45\delta_1^2+120\delta_1^3-210\delta_1^4+252\delta_1^5-210\delta_1^6+120\delta_1^7-45\delta_1^8+10\delta_1^9-\delta_1^{10}\bigr).
\end{align*}

We highlight the lowest-degree terms in the $K$-Chern class relations. For $\overline{\sM}_2$, let $\Delta_1$ be the divisor of curves with a separating node and let $\beta=c_1^K(\cO(\Delta_1))$. Edidin--Hu's presentation \cite[Theorem~1.2, Corollary~5.8, and Theorem~5.11]{EH24} can be written in $K$-Chern classes as
\[
K_0(\overline{\sM}_2)\cong
\frac{\bZ[\delta_1,\delta_2,\beta]}
{(I_{\Delta_1}\cap I_{\overline{\sM}_2\setminus\Delta_1})+\beta I_{\Delta_1}},
\]
where the two ideals are
\begin{align*}
&I_{\Delta_1}=\bigl(\\*
&\quad \colorbox{HighlightGold}{$\displaystyle\textcolor{RoyalPurple}{2\delta_1+2\beta}$}-\delta_1^{2}-4\delta_1\beta-\beta^{2}+2\delta_1^{2}\beta+2\delta_1\beta^{2}-\delta_1^{2}\beta^{2},\\[3pt]
&\quad \colorbox{HighlightGold}{$\displaystyle\textcolor{RoyalPurple}{2\delta_1+2\beta}$}-\delta_1^{2}-3\delta_1\beta+\delta_1^{2}\beta-\delta_1\delta_2-\delta_2\beta+\delta_1\delta_2\beta,\\[3pt]
&\quad \colorbox{HighlightGold}{$\displaystyle\textcolor{RoyalPurple}{-28\delta_1^{2}-4\delta_1\beta+48\delta_2}$}+128\delta_1^{3}+36\delta_1^{2}\beta-96\delta_1\delta_2-299\delta_1^{4}-137\delta_1^{3}\beta+104\delta_1^{2}\delta_2-124\delta_2^{2}+445\delta_1^{5}+304\delta_1^{4}\beta\\*
&\quad -66\delta_1^{3}\delta_2+226\delta_1\delta_2^{2}-456\delta_1^{6}-446\delta_1^{5}\beta+24\delta_1^{4}\delta_2-189\delta_1^{2}\delta_2^{2}+120\delta_2^{3}+329\delta_1^{7}+456\delta_1^{6}\beta-4\delta_1^{5}\delta_2+80\delta_1^{3}\delta_2^{2}\\*
&\quad -172\delta_1\delta_2^{3}-165\delta_1^{8}-329\delta_1^{7}\beta-14\delta_1^{4}\delta_2^{2}+96\delta_1^{2}\delta_2^{3}-55\delta_2^{4}+55\delta_1^{9}+165\delta_1^{8}\beta-20\delta_1^{3}\delta_2^{3}+54\delta_1\delta_2^{4}-11\delta_1^{10}-55\delta_1^{9}\beta\\*
&\quad -15\delta_1^{2}\delta_2^{4}+12\delta_2^{5}+\delta_1^{11}+11\delta_1^{10}\beta-6\delta_1\delta_2^{5}-\delta_1^{11}\beta-\delta_2^{6},\\[3pt]
&\quad \colorbox{HighlightGold}{$\displaystyle\textcolor{RoyalPurple}{28\delta_1^{2}+4\delta_1\beta-48\delta_2}$}-124\delta_1^{3}-32\delta_1^{2}\beta+96\delta_1\delta_2+283\delta_1^{4}+117\delta_1^{3}\beta-104\delta_1^{2}\delta_2+700\delta_2^{2}-416\delta_1^{5}-259\delta_1^{4}\beta+66\delta_1^{3}\delta_2\\*
&\quad -2530\delta_1\delta_2^{2}+425\delta_1^{6}+386\delta_1^{5}\beta-24\delta_1^{4}\delta_2+4845\delta_1^{2}\delta_2^{2}-2520\delta_2^{3}-309\delta_1^{7}-405\delta_1^{6}\beta+4\delta_1^{5}\delta_2-5984\delta_1^{3}\delta_2^{2}+8572\delta_1\delta_2^{3}\\*
&\quad +158\delta_1^{8}+302\delta_1^{7}\beta+5030\delta_1^{4}\delta_2^{2}-14256\delta_1^{2}\delta_2^{3}+4235\delta_2^{4}-54\delta_1^{9}-157\delta_1^{8}\beta-2880\delta_1^{5}\delta_2^{2}+14420\delta_1^{3}\delta_2^{3}-12594\delta_1\delta_2^{4}\\*
&\quad +11\delta_1^{10}+54\delta_1^{9}\beta+1080\delta_1^{6}\delta_2^{2}-9360\delta_1^{4}\delta_2^{3}+17305\delta_1^{2}\delta_2^{4}-3992\delta_2^{5}-\delta_1^{11}-11\delta_1^{10}\beta-240\delta_1^{7}\delta_2^{2}+3840\delta_1^{5}\delta_2^{3}\\*
&\quad -13680\delta_1^{3}\delta_2^{4}+9956\delta_1\delta_2^{5}+\delta_1^{11}\beta+24\delta_1^{8}\delta_2^{2}-912\delta_1^{6}\delta_2^{3}+6450\delta_1^{4}\delta_2^{4}-10824\delta_1^{2}\delta_2^{5}+2274\delta_2^{6}+96\delta_1^{7}\delta_2^{3}-1700\delta_1^{5}\delta_2^{4}\\*
&\quad +6286\delta_1^{3}\delta_2^{5}-4546\delta_1\delta_2^{6}+194\delta_1^{6}\delta_2^{4}-1924\delta_1^{4}\delta_2^{5}+3669\delta_1^{2}\delta_2^{6}-800\delta_2^{7}+246\delta_1^{5}\delta_2^{5}-1396\delta_1^{3}\delta_2^{6}+1200\delta_1\delta_2^{7}+209\delta_1^{4}\delta_2^{6}\\*
&\quad -640\delta_1^{2}\delta_2^{7}+170\delta_2^{8}+120\delta_1^{3}\delta_2^{7}-170\delta_1\delta_2^{8}+45\delta_1^{2}\delta_2^{8}-20\delta_2^{9}+10\delta_1\delta_2^{9}+\delta_2^{10},\\[3pt]
&\quad \colorbox{HighlightGold}{$\displaystyle\textcolor{RoyalPurple}{24\delta_1^{2}-48\delta_2}$}-96\delta_1^{3}+72\delta_1\delta_2+194\delta_1^{4}-8\delta_1^{2}\delta_2+964\delta_2^{2}-246\delta_1^{5}-128\delta_1^{3}\delta_2-3796\delta_1\delta_2^{2}+209\delta_1^{6}+222\delta_1^{4}\delta_2\\*
&\quad +8099\delta_1^{2}\delta_2^{2}-4390\delta_2^{3}-120\delta_1^{7}-205\delta_1^{5}\delta_2-11455\delta_1^{3}\delta_2^{2}+16867\delta_1\delta_2^{3}+45\delta_1^{8}+120\delta_1^{6}\delta_2+11403\delta_1^{4}\delta_2^{2}-32387\delta_1^{2}\delta_2^{3}\\*
&\quad +9218\delta_2^{4}-10\delta_1^{9}-45\delta_1^{7}\delta_2-8085\delta_1^{5}\delta_2^{2}+38975\delta_1^{3}\delta_2^{3}-31813\delta_1\delta_2^{4}+\delta_1^{10}+10\delta_1^{8}\delta_2+4015\delta_1^{6}\delta_2^{2}-31405\delta_1^{4}\delta_2^{3}\\*
&\quad +52180\delta_1^{2}\delta_2^{4}-10933\delta_2^{5}-\delta_1^{9}\delta_2-1331\delta_1^{7}\delta_2^{2}+17083\delta_1^{5}\delta_2^{3}-51315\delta_1^{3}\delta_2^{4}+32642\delta_1\delta_2^{5}+265\delta_1^{8}\delta_2^{2}-6059\delta_1^{6}\delta_2^{3}\\*
&\quad +32125\delta_1^{4}\delta_2^{4}-44239\delta_1^{2}\delta_2^{5}+7994\delta_2^{6}-24\delta_1^{9}\delta_2^{2}+1271\delta_1^{7}\delta_2^{3}-12654\delta_1^{5}\delta_2^{4}+34161\delta_1^{3}\delta_2^{5}-19957\delta_1\delta_2^{6}-120\delta_1^{8}\delta_2^{3}\\*
&\quad +2879\delta_1^{6}\delta_2^{4}-15644\delta_1^{4}\delta_2^{5}+21448\delta_1^{2}\delta_2^{6}-3739\delta_2^{7}-290\delta_1^{7}\delta_2^{4}+3991\delta_1^{5}\delta_2^{5}-12222\delta_1^{3}\delta_2^{6}+7476\delta_1\delta_2^{7}-440\delta_1^{6}\delta_2^{5}\\*
&\quad +3654\delta_1^{4}\delta_2^{6}-5978\delta_1^{2}\delta_2^{7}+1122\delta_2^{8}-455\delta_1^{5}\delta_2^{6}+2240\delta_1^{3}\delta_2^{7}-1683\delta_1\delta_2^{8}-329\delta_1^{4}\delta_2^{7}+891\delta_1^{2}\delta_2^{8}-209\delta_2^{9}-165\delta_1^{3}\delta_2^{8}\\*
&\quad +209\delta_1\delta_2^{9}-55\delta_1^{2}\delta_2^{9}+22\delta_2^{10}-11\delta_1\delta_2^{10}-\delta_2^{11}\bigr),\\[8pt]
&I_{\overline{\sM}_2\setminus\Delta_1}=\bigl(\\*
&\quad \colorbox{HighlightGold}{$\displaystyle\textcolor{RoyalPurple}{24\delta_1^{2}-48\delta_2}$}-120\delta_1^{3}+364\delta_1\delta_2+266\delta_1^{4}-1270\delta_1^{2}\delta_2+700\delta_2^{2}-344\delta_1^{5}+2530\delta_1^{3}\delta_2-3365\delta_1\delta_2^{2}+285\delta_1^{6}-3186\delta_1^{4}\delta_2\\*
&\quad +7807\delta_1^{2}\delta_2^{2}-2520\delta_2^{3}-155\delta_1^{7}+2650\delta_1^{5}\delta_2-10887\delta_1^{3}\delta_2^{2}+9724\delta_1\delta_2^{3}+54\delta_1^{8}-1464\delta_1^{6}\delta_2+9789\delta_1^{4}\delta_2^{2}-17528\delta_1^{2}\delta_2^{3}\\*
&\quad +4235\delta_2^{4}-11\delta_1^{9}+520\delta_1^{7}\delta_2-5760\delta_1^{5}\delta_2^{2}+18576\delta_1^{3}\delta_2^{3}-13474\delta_1\delta_2^{4}+\delta_1^{10}-108\delta_1^{8}\delta_2+2156\delta_1^{6}\delta_2^{2}-12300\delta_1^{4}\delta_2^{3}\\*
&\quad +19234\delta_1^{2}\delta_2^{4}-3992\delta_2^{5}+10\delta_1^{9}\delta_2-468\delta_1^{7}\delta_2^{2}+5040\delta_1^{5}\delta_2^{3}-15455\delta_1^{3}\delta_2^{4}+10356\delta_1\delta_2^{5}+45\delta_1^{8}\delta_2^{2}-1176\delta_1^{6}\delta_2^{3}\\*
&\quad +7280\delta_1^{4}\delta_2^{4}-11460\delta_1^{2}\delta_2^{5}+2274\delta_2^{6}+120\delta_1^{7}\delta_2^{3}-1890\delta_1^{5}\delta_2^{4}+6664\delta_1^{3}\delta_2^{5}-4654\delta_1\delta_2^{6}+210\delta_1^{6}\delta_2^{4}-2016\delta_1^{4}\delta_2^{5}\\*
&\quad +3780\delta_1^{2}\delta_2^{6}-800\delta_2^{7}+252\delta_1^{5}\delta_2^{5}-1428\delta_1^{3}\delta_2^{6}+1216\delta_1\delta_2^{7}+210\delta_1^{4}\delta_2^{6}-648\delta_1^{2}\delta_2^{7}+170\delta_2^{8}+120\delta_1^{3}\delta_2^{7}-171\delta_1\delta_2^{8}\\*
&\quad +45\delta_1^{2}\delta_2^{8}-20\delta_2^{9}+10\delta_1\delta_2^{9}+\delta_2^{10},\\[3pt]
&\quad \colorbox{HighlightGold}{$\displaystyle\textcolor{RoyalPurple}{24\delta_1^{2}-48\delta_2}$}-96\delta_1^{3}+296\delta_1\delta_2+170\delta_1^{4}-856\delta_1^{2}\delta_2+484\delta_2^{2}-174\delta_1^{5}+1408\delta_1^{3}\delta_2-1996\delta_1\delta_2^{2}+111\delta_1^{6}-1434\delta_1^{4}\delta_2\\*
&\quad +3891\delta_1^{2}\delta_2^{2}-1350\delta_2^{3}-44\delta_1^{7}+931\delta_1^{5}\delta_2-4439\delta_1^{3}\delta_2^{2}+4415\delta_1\delta_2^{3}+10\delta_1^{8}-378\delta_1^{6}\delta_2+3140\delta_1^{4}\delta_2^{2}-6553\delta_1^{2}\delta_2^{3}\\*
&\quad +1772\delta_2^{4}-\delta_1^{9}+88\delta_1^{7}\delta_2-1365\delta_1^{5}\delta_2^{2}+5490\delta_1^{3}\delta_2^{3}-4660\delta_1\delta_2^{4}-9\delta_1^{8}\delta_2+336\delta_1^{6}\delta_2^{2}-2695\delta_1^{4}\delta_2^{3}+5275\delta_1^{2}\delta_2^{4}-1286\delta_2^{5}\\*
&\quad -36\delta_1^{7}\delta_2^{2}+728\delta_1^{5}\delta_2^{3}-3150\delta_1^{3}\delta_2^{4}+2651\delta_1\delta_2^{5}-84\delta_1^{6}\delta_2^{3}+980\delta_1^{4}\delta_2^{4}-2184\delta_1^{2}\delta_2^{5}+546\delta_2^{6}-126\delta_1^{5}\delta_2^{4}+840\delta_1^{3}\delta_2^{5}\\*
&\quad -833\delta_1\delta_2^{6}-126\delta_1^{4}\delta_2^{5}+448\delta_1^{2}\delta_2^{6}-135\delta_2^{7}-84\delta_1^{3}\delta_2^{6}+136\delta_1\delta_2^{7}-36\delta_1^{2}\delta_2^{7}+18\delta_2^{8}-9\delta_1\delta_2^{8}-\delta_2^{9},\\[3pt]
&\quad \colorbox{HighlightGold}{$\displaystyle\textcolor{RoyalPurple}{-108\delta_1^{2}\delta_2+576\delta_2^{2}}$}+378\delta_1^{3}\delta_2-2304\delta_1\delta_2^{2}-606\delta_1^{4}\delta_2+4651\delta_1^{2}\delta_2^{2}-2400\delta_2^{3}+570\delta_1^{5}\delta_2-5889\delta_1^{3}\delta_2^{2}+8400\delta_1\delta_2^{3}\\*
&\quad -334\delta_1^{6}\delta_2+4927\delta_1^{4}\delta_2^{2}-13760\delta_1^{2}\delta_2^{3}+4180\delta_2^{4}+120\delta_1^{7}\delta_2-2727\delta_1^{5}\delta_2^{2}+13400\delta_1^{3}\delta_2^{3}-12540\delta_1\delta_2^{4}-24\delta_1^{8}\delta_2\\*
&\quad +960\delta_1^{6}\delta_2^{2}-8224\delta_1^{4}\delta_2^{3}+16744\delta_1^{2}\delta_2^{4}-3980\delta_2^{5}+2\delta_1^{9}\delta_2-194\delta_1^{7}\delta_2^{2}+3136\delta_1^{5}\delta_2^{3}-12588\delta_1^{3}\delta_2^{4}+9950\delta_1\delta_2^{5}+17\delta_1^{8}\delta_2^{2}\\*
&\quad -680\delta_1^{6}\delta_2^{3}+5554\delta_1^{4}\delta_2^{4}-10496\delta_1^{2}\delta_2^{5}+2273\delta_2^{6}+64\delta_1^{7}\delta_2^{3}-1350\delta_1^{5}\delta_2^{4}+5794\delta_1^{3}\delta_2^{5}-4546\delta_1\delta_2^{6}+140\delta_1^{6}\delta_2^{4}\\*
&\quad -1660\delta_1^{4}\delta_2^{5}+3567\delta_1^{2}\delta_2^{6}-800\delta_2^{7}+196\delta_1^{5}\delta_2^{5}-1294\delta_1^{3}\delta_2^{6}+1200\delta_1\delta_2^{7}+182\delta_1^{4}\delta_2^{6}-624\delta_1^{2}\delta_2^{7}+170\delta_2^{8}+112\delta_1^{3}\delta_2^{7}\\*
&\quad -170\delta_1\delta_2^{8}+44\delta_1^{2}\delta_2^{8}-20\delta_2^{9}+10\delta_1\delta_2^{9}+\delta_2^{10},\\[3pt]
&\quad \colorbox{HighlightGold}{$\displaystyle\textcolor{RoyalPurple}{120\delta_1^3}$}-600\delta_1^4+1450\delta_1^5-2200\delta_1^6+2299\delta_1^7-1717\delta_1^8+923\delta_1^9-351\delta_1^{10}+90\delta_1^{11}-14\delta_1^{12}+\delta_1^{13}\bigr).
\end{align*}
\endgroup

\Needspace{7\baselineskip}
For $g=3$, the first relation in bundle class generators $\sigma,\lambda^{\pm1}$ is
\[
-\sigma^{6}\,\lambda^{4}+5\,\sigma^{5}\,\lambda^{4}-5\,\sigma^{4}\,\lambda^{4}-6\,\sigma^{3}\,\lambda^{4}+\sigma^{3}\,\lambda^{2}+7\,\sigma^{2}\,\lambda^{4}-\sigma^{2}\,\lambda^{2}+2\,\sigma\,\lambda^{4}-2\,\sigma\,\lambda^{2}-\lambda^{4}+1
\]
and the second is
\[
-\sigma^{6}\,\lambda^{4}+4\,\sigma^{5}\,\lambda^{4}-\sigma^{4}\,\lambda^{4}-8\,\sigma^{3}\,\lambda^{4}+\sigma^{3}\,\lambda^{2}+2\,\sigma^{2}\,\lambda^{4}+4\,\sigma\,\lambda^{4}-2\,\sigma\,\lambda^{2}+\sigma-\lambda^{2}+1.
\]

For $g=3$, \Cref{prop:sectorpoint} identifies the $28$ components of $\sI\sH_3$ with the points of $\Spec(K_0(\sH_3)\otimes\overline{\bQ})$. Each dot in \Cref{fig:g3-specialization} is labelled by the values of $\sigma=[\cE]$ and $\lambda=[\cL]$, which at $(C,\tau)\in\sI\sH_3$ are the traces of $\tau$ on the respective fibers at $[C]$.

For a pair $(C,\tau)$ on a component other than those for $1$ and $\iota$, choose coordinates in which $C$ has equation $y^2=f(x,1)$ and $\tau$ induces $x\mapsto\rho x$ on $\bP^1$. The expressions for $f$ in \Cref{fig:g3-specialization} use a square-free binary form $u$ satisfying $u(1,0)u(0,1)\ne0$, with $\deg u$ determined by $\deg f=8$. If $\operatorname{diag}(\rho\beta,\beta)$ represents $\tau$, then $(\sigma,\lambda)=(1+\rho+\rho^{-1},\rho^2\beta^4)$ and $\tau(x,y)=(\rho x,\beta^{-4}y)=(\rho x,\rho^2\lambda^{-1}y)$. The value of $\sigma$ determines $\rho$ up to taking its inverse, with the choice $\rho=\lambda^{-4}$ when one fixed point is a root and the other is not.

Let $\mathfrak m=\ker(\operatorname{rank}\colon K_0(\sH_3)\otimes\bF_2\to\bF_2)$, and let $F$ be the image of the gamma-filtration on $K_0(\sH_3)\otimes\bF_2$. The twelve purple points have both $\rho$ and $\lambda$ of order a power of $2$, so they specialize to $\mathfrak m$ by \Cref{prop:local-lengths}. Thus, $\dim_{\bF_2}\bigl(K_0(\sH_3)\otimes\bF_2\bigr)_{\mathfrak m}=12$, as is also visible in the monomial basis of $\operatorname{gr}_F\bigl(K_0(\sH_3)\otimes\bF_2\bigr)$ from \Cref{prop:g3mod2}.

\begin{figure}[!htbp]
\centering
\begin{tikzpicture}[x=1cm,y=1cm,font=\small]
  \node[font=\small\bfseries] at (7.75,16.10)
    {(a) The $28$ components of $\sI\sH_3$};
  \node at (7.75,15.55)
    {$\pi_0(\sI\sH_3)\ \longleftrightarrow\ \Spec(K_0(\sH_3)\otimes\overline{\bQ})$};
  \node at (7.75,14.98)
    {Dot labels: $(\sigma,\lambda)$, the values of $[\cE]$ and $[\cL]$ at each point.};

  \node[anchor=west,inner xsep=0pt] at (0.15,14.30) {Identity: $\sH_3$};
  \node[anchor=west,inner xsep=0pt] at (7.50,14.30)
    {Hyperelliptic involution: $\sH_3$};
  \foreach \x/\coord in {4.45/{(3,1)},13.65/{(3,-1)}} {
    \fill[RoyalPurple] (\x,14.30) circle (2.1pt);
    \node[anchor=west,inner xsep=0pt,text=RoyalPurple]
      at (\x+0.20,14.30) {$\coord$};
  }

  \node[anchor=west,inner xsep=0pt,font=\small\bfseries] at (0.15,13.65)
    {Neither $(0:1)$ nor $(1:0)$ is a root};
  \foreach \y/\form in {
    13.08/{u(x_0^2,x_1^2)},12.49/{u(x_0^4,x_1^4)},11.90/{u(x_0^8,x_1^8)}}
    \node[anchor=west,inner xsep=0pt] at (0.15,\y) {$f=\form$};
  \foreach \x/\y/\coord in {
    4.45/13.08/{(-1,1)},7.15/13.08/{(-1,-1)},
    4.45/12.49/{(1,1)},7.15/12.49/{(1,-1)},
    4.45/11.90/{(1+\sqrt2,i)},7.15/11.90/{(1+\sqrt2,-i)},
    9.85/11.90/{(1-\sqrt2,i)},12.55/11.90/{(1-\sqrt2,-i)}} {
    \fill[RoyalPurple] (\x,\y) circle (2.1pt);
    \node[anchor=west,inner xsep=0pt,text=RoyalPurple]
      at (\x+0.20,\y) {$\coord$};
  }

  \node[anchor=west,inner xsep=0pt,font=\small\bfseries] at (0.15,11.07)
    {Both $(0:1)$ and $(1:0)$ are roots};
  \foreach \y/\form in {
    10.49/{x_0x_1u(x_0^2,x_1^2)},9.90/{x_0x_1u(x_0^3,x_1^3)},9.31/{x_0x_1u(x_0^6,x_1^6)}}
    \node[anchor=west,inner xsep=0pt] at (0.15,\y) {$f=\form$};
  \foreach \x/\coord in {4.45/{(-1,i)},7.15/{(-1,-i)}} {
    \fill[RoyalPurple] (\x,10.49) circle (2.1pt);
    \node[anchor=west,inner xsep=0pt,text=RoyalPurple]
      at (\x+0.20,10.49) {$\coord$};
  }
  \foreach \x/\y/\coord in {
    4.45/9.90/{(0,1)},7.15/9.90/{(0,-1)},
    4.45/9.31/{(2,i)},7.15/9.31/{(2,-i)}} {
    \fill[black!50] (\x,\y) circle (1.8pt);
    \node[anchor=west,inner xsep=0pt,text=black!65]
      at (\x+0.20,\y) {$\coord$};
  }

  \node[anchor=west,inner xsep=0pt,font=\small\bfseries] at (0.15,8.48)
    {$(0:1)$ is a root and $(1:0)$ is not};
  \node[anchor=west,inner xsep=0pt] at (8.00,8.48)
    {$f=x_0u(x_0^7,x_1^7)$};
  \foreach \x/\y/\coord in {
    0.15/7.76/{(1+2\cos(2\pi/7),e^{3\pi i/7})},
    5.40/7.76/{(1+2\cos(4\pi/7),e^{\pi i/7})},
    10.65/7.76/{(1+2\cos(6\pi/7),e^{2\pi i/7})},
    0.15/7.15/{(1+2\cos(2\pi/7),e^{-3\pi i/7})},
    5.40/7.15/{(1+2\cos(4\pi/7),e^{-\pi i/7})},
    10.65/7.15/{(1+2\cos(6\pi/7),e^{-2\pi i/7})},
    0.15/6.54/{(1+2\cos(2\pi/7),-e^{3\pi i/7})},
    5.40/6.54/{(1+2\cos(4\pi/7),-e^{\pi i/7})},
    10.65/6.54/{(1+2\cos(6\pi/7),-e^{2\pi i/7})},
    0.15/5.93/{(1+2\cos(2\pi/7),-e^{-3\pi i/7})},
    5.40/5.93/{(1+2\cos(4\pi/7),-e^{-\pi i/7})},
    10.65/5.93/{(1+2\cos(6\pi/7),-e^{-2\pi i/7})}} {
    \fill[black!50] (\x,\y) circle (1.8pt);
    \node[anchor=west,inner xsep=0pt,text=black!65]
      at (\x+0.20,\y) {$\coord$};
  }
  \draw[black!18,line width=0.4pt] (0.00,5.38) -- (15.55,5.38);

  \node[font=\small\bfseries] at (3.40,4.86) {(b) Specialization modulo $2$};
  \node at (3.40,4.32) {$\Spec(K_0(\sH_3)\otimes\bF_2)$};
  \node[RoyalPurple] at (3.40,3.67) {$(\sigma,\lambda)=(1,1)$};
  \fill[RoyalPurple!12] (3.40,3.08) circle (0.27cm);
  \fill[RoyalPurple] (3.40,3.08) circle (3.5pt);
  \node at (3.40,2.47) {$\mathfrak m=(\tau_1,\kappa_2)$};
  \node[RoyalPurple] at (3.40,1.96) {$\dim_{\bF_2}(K_0(\sH_3)\otimes\bF_2)_{\mathfrak m}=12$};
  \node[align=center,text width=6.1cm,text=black!65] at (3.40,1.12)
    {The $16$ gray points specialize elsewhere.};

  \draw[->,RoyalPurple,densely dashed,line width=0.7pt]
    (5.23,3.08) -- (8.57,3.08)
    node[midway,above=5pt,text=black,font=\footnotesize]
    {associated graded};
  \node[font=\small\bfseries] at (11.90,4.86) {(c) The $12$ basis monomials};
  \node at (11.90,4.32)
    {$\operatorname{gr}_F(K_0(\sH_3)\otimes\bF_2)$};

  \begin{scope}[shift={(9.10,1.10)},x=1.40cm,y=0.54cm]
    \foreach \i in {0,1,2,3}
      \fill[RoyalPurple!15] (\i,0) rectangle +(1,1);
    \foreach \i in {0,1}
      \foreach \j in {1,2,3,4}
        \fill[RoyalPurple!15] (\i,\j) rectangle +(1,1);
    \draw[step=1,black!18,line width=0.3pt] (0,0) grid (4,5);
    \draw[RoyalPurple,line width=0.85pt]
      (0,0) -- (4,0) -- (4,1) -- (2,1) -- (2,5) -- (0,5) -- cycle;
    \foreach \x/\y/\monomial in {
      0.5/0.5/{1},1.5/0.5/{\tau_1},2.5/0.5/{\tau_1^2},3.5/0.5/{\tau_1^3},
      0.5/1.5/{\kappa_2},1.5/1.5/{\tau_1\kappa_2},
      0.5/2.5/{\kappa_2^2},1.5/2.5/{\tau_1\kappa_2^2},
      0.5/3.5/{\kappa_2^3},1.5/3.5/{\tau_1\kappa_2^3},
      0.5/4.5/{\kappa_2^4},1.5/4.5/{\tau_1\kappa_2^4}}
      \node at (\x,\y) {$\monomial$};
  \end{scope}
  \node at (3.40,0.10) {$CH^\bullet(\sH_3)\otimes\bF_2$};
  \node at (3.40,-0.48) {$\cong\bF_2[t,c_2,c_3]$};
  \draw[->>,RoyalPurple,line width=0.7pt]
    (6.10,0.10) -- (8.95,0.10)
    node[midway,above=5pt,text=black,font=\footnotesize]
    {$(t,c_2,c_3)\longmapsto(\tau_1,\kappa_2,0)$};
  \node at (11.90,0.10)
    {$\operatorname{gr}_F(K_0(\sH_3)\otimes\bF_2)$};
  \node at (11.90,-0.48)
    {$\cong\bF_2[\tau_1,\kappa_2]/(\tau_1^4,\tau_1^2\kappa_2,\kappa_2^5)$};
\end{tikzpicture}
\caption{(a) Each dot represents a point of $\Spec(K_0(\sH_3)\otimes\overline{\bQ})$, labelled by the coordinates $(\sigma,\lambda)$. Each point corresponds to one component of $\sI\sH_3$. (b) Reducing coefficients modulo $2$ sends the twelve purple points to $(1,1)$, where the local ring has length $12$. (c) The twelve monomials form a basis of $\operatorname{gr}_F(K_0(\sH_3)\otimes\bF_2)$, which is a quotient of $CH^\bullet(\sH_3)\otimes\bF_2$.}
\label{fig:g3-specialization}
\end{figure}

\subsection{Proofs of the $K$-Chern class presentations}\label{app:chernproofs}

We prove \Cref{thm:introfc} by expanding the relations in bundle classes. In even genus, the identities
\[
[\cV^\vee]=2-\delta_1-\delta_2,
\qquad
[\det\cV^\vee]=1-\delta_1
\]
give $\nu=(2-\delta_1-\delta_2)/(1-\delta_1)$ and $\lambda=(1-\delta_1)^{-1}$. In odd genus, \Cref{lem:cEchern} gives $\sigma=3-\kappa_2$ and $\lambda=(1-\tau_1)^{-1}$.

The first two defining polynomials $\Phi_g,\Psi_g$ in \eqref{eq:K0-Hg-chern-class} and \eqref{eq:odd-chern-binomial} satisfy
\begin{equation}\label{eq:chern-coefficient-identities}
\begin{aligned}
\Phi_g&=\lambda^{-(2g+2)}F_g^{\mathrm{even}},&
\Psi_g&=\lambda^{-(2g+2)}(G_g^{\mathrm{even}}-F_g^{\mathrm{even}})
&& (g\text{ even}),\\
\Phi_g&=\lambda^{-4}F_g^{\mathrm{odd}},&
\Psi_g&=\lambda^{-4}(G_g^{\mathrm{odd}}-2F_g^{\mathrm{odd}})
&& (g\text{ odd}).
\end{aligned}
\end{equation}
We use the binomial expansions
\begin{equation}\label{eq:chern-polynomial-expansions}
\begin{aligned}
P_k(2-\delta_1-\delta_2,1-\delta_1)
&=\sum_{\substack{i,j\ge0\\i+j\le k}}(-1)^{i+j}
\binom{i+j}{j}\binom{k+j+1}{i+2j+1}\delta_1^i\delta_2^j,\\
Q_k(3-\kappa_2)
&=\sum_{j=0}^k(-1)^j\left[\binom{k+j+1}{2j+1}+\binom{k+j}{2j+1}\right]\kappa_2^j.
\end{aligned}
\end{equation}

\begin{proof}[Proof of \Cref{thm:introfc}]
We first prove the binomial expansions in \Cref{eq:chern-polynomial-expansions}. The recurrences for $P_k$ and $Q_k$ give
\[
\begin{aligned}
\sum_{k\ge0}P_k(2-\delta_1-\delta_2,1-\delta_1)u^k
&=\frac{1}{(1-u)^2+\delta_1u(1-u)+\delta_2u},\\
\sum_{k\ge0}Q_k(3-\kappa_2)u^k
&=\frac{1+u}{(1-u)^2+\kappa_2u}.
\end{aligned}
\]
In the first series, the coefficient of $\delta_1^i\delta_2^j$ is
\[
(-1)^{i+j}\binom{i+j}{j}\frac{u^{i+j}}{(1-u)^{i+2j+2}},
\]
whose coefficient of $u^k$ is $(-1)^{i+j}\binom{i+j}{j}\binom{k+j+1}{i+2j+1}$. In the second series, the coefficient of $\kappa_2^j$ is $(-1)^j u^j(1+u)/(1-u)^{2j+2}$, whose coefficient of $u^k$ is $(-1)^j[\binom{k+j+1}{2j+1}+\binom{k+j}{2j+1}]$.

Using $P_k(\nu\lambda^{-1},\lambda^{-1})=\lambda^{-k}P_k(\nu,\lambda)$, we substitute the binomial expansions into the right-hand sides of \eqref{eq:chern-coefficient-identities}. In even genus, applying $XP_k(X,Y)=P_{k+1}(X,Y)+YP_{k-1}(X,Y)$ and Pascal's identity, then expanding the powers of $1-\delta_1$, gives the coefficients in \Cref{eq:K0-Hg-chern-class}. In odd genus, the recurrence gives
\[
\begin{aligned}
\lambda^{-4}(G_g^{\mathrm{odd}}-2F_g^{\mathrm{odd}})
&=(2-\kappa_2)(1-\tau_1)^4
+(1-\tau_1)^2\bigl(Q_{g-2}(3-\kappa_2)-Q_g(3-\kappa_2)\bigr)\\
&\quad+Q_{2g}(3-\kappa_2)-Q_{2g-1}(3-\kappa_2).
\end{aligned}
\]
Expanding the powers of $1-\tau_1$ and applying Pascal's identity to \Cref{eq:chern-polynomial-expansions} gives the coefficients in \Cref{eq:odd-chern-binomial}.

For even $g$, the binomial expansion in \Cref{eq:chern-polynomial-expansions} gives, modulo $(\delta_1^3,\delta_1\delta_2,\delta_2^2)$,
\[
P_k(2-\delta_1-\delta_2,1-\delta_1)
\equiv k+1-\binom{k+1}{2}\delta_1
+\binom{k+1}{3}\delta_1^2-\binom{k+2}{3}\delta_2.
\]
Substituting in the formulas for $\Phi_g$ and $\Psi_g$ gives
\[
\begin{aligned}
\Phi_g&\equiv-(4g+2)\delta_1
&&\pmod{(\delta_1^2,\delta_2)},\\
\Psi_g&\equiv g(g-1)\delta_1^2-4g(g+1)\delta_2
&&\pmod{(\delta_1^3,\delta_1\delta_2,\delta_2^2)}.
\end{aligned}
\]
Expanding $1-(1-\delta_1)^{4g+2}$ by the binomial theorem gives the third relation in \eqref{eq:K0-Hg-chern-class}.

For odd $g$, \Cref{eq:chern-polynomial-expansions} gives
\[
Q_k(3-\kappa_2)\equiv2k+1-\frac{k(k+1)(2k+1)}6\kappa_2
\pmod{\kappa_2^2}.
\]
Substituting in the formulas for $\Phi_g$ and $\Psi_g$ gives
\[
\begin{aligned}
\Phi_g&\equiv-(8g+4)\tau_1
&&\pmod{(\tau_1^2,\kappa_2)},\\
\Psi_g&\equiv8\tau_1^2-2g(g+1)\kappa_2
&&\pmod{(\tau_1^3,\tau_1\kappa_2,\kappa_2^2)}.
\end{aligned}
\]
Expanding $1-(1-\tau_1)^{8g+4}$ by the binomial theorem gives the third relation in \eqref{eq:odd-chern-binomial}.

Replacing $G_g^{\mathrm{even}}$ by $G_g^{\mathrm{even}}-F_g^{\mathrm{even}}$, or $G_g^{\mathrm{odd}}$ by $G_g^{\mathrm{odd}}-2F_g^{\mathrm{odd}}$, preserves the respective relation ideals. Finally, $(1-\delta_1)^{4g+2}=1$ and $(1-\tau_1)^{8g+4}=1$ make $1-\delta_1$ and $1-\tau_1$ invertible. Thus the polynomial quotients in \Cref{thm:introfc} agree with the Laurent presentations of \Cref{thm:even,thm:odd}.
\end{proof}

\begin{proof}[Verification of \Cref{eq:even-mod-two-leading,eq:odd-mod-two-leading}]
All coefficients are reduced modulo $2$. For even $g$, the square of $\delta_1+\delta_2$ vanishes modulo $(\delta_1^2,\delta_2^2)$, so only the constant and linear terms in the first argument of $P_k$ remain. Its defining sum gives
\[
\begin{aligned}
P_{2g-1}(\delta_1+\delta_2,1+\delta_1)
&=P_{4g-1}(\delta_1+\delta_2,1+\delta_1)=0,\\
P_{2g-2}(\delta_1+\delta_2,1+\delta_1)&=1+\delta_1,
\qquad P_{4g}(\delta_1+\delta_2,1+\delta_1)=1.
\end{aligned}
\]
Substitution in \Cref{eq:chern-coefficient-identities} proves \Cref{eq:even-mod-two-leading}.

For odd $g$, the binomial expansions in \Cref{eq:chern-polynomial-expansions}, reduced modulo $2$, give
\[
\begin{aligned}
Q_{g-1}(3-\kappa_2)&\equiv1+\frac{g-1}{2}\kappa_2
&&\pmod{\kappa_2^2},\\
Q_{2g}(3-\kappa_2)-Q_{2g-1}(3-\kappa_2)&\equiv\kappa_2^2
&&\pmod{\kappa_2^3}.
\end{aligned}
\]
The recurrence $Q_{g-2}-Q_g=(1-\sigma)Q_{g-1}$ in characteristic two gives
\[
\begin{aligned}
\Psi_g&=\kappa_2\bigl(1+\tau_1^4+(1+\tau_1^2)Q_{g-1}(3-\kappa_2)\bigr)\\
&\quad+Q_{2g}(3-\kappa_2)-Q_{2g-1}(3-\kappa_2),
\end{aligned}
\]
which proves the second congruence in \Cref{eq:odd-mod-two-leading}. For $\Phi_g$, the coefficient of $\tau_1^2$ is $4g$, and that of $\kappa_2$ is the sum of $g(2g^2+1)/3$ and $g(2g+1)(4g+1)/3$, both odd. Its constant term also vanishes, proving the first congruence.
\end{proof}

\printbibliography

\bigskip
\noindent\begin{minipage}{\textwidth}
\small
\textsc{Michael Ruofan Zeng}\\
Department of Mathematics, University of Washington\\
Box 354350, C-138 Padelford\\
Seattle, WA 98195-4350, USA\\
Email: \texttt{zengrf@uw.edu}\\
Website: \href{https://zengrf.github.io}{\texttt{zengrf.github.io}}
\end{minipage}

\clearpage
\end{CJK*}
\end{document}